\def\version{}%{\tiny Version January 20, 2020 (typeset: \today)}}

\documentclass[reqno,twoside,11pt]{amsart}
\usepackage{cite}
\usepackage{amsmath}
\usepackage{amsfonts}
\usepackage{amssymb}
\usepackage{epsfig}
\usepackage{verbatim}
\usepackage{color}
\usepackage{eucal,mathrsfs,dsfont}

\IfFileExists{epsf.def}{\input epsf.def}{\usepackage{epsf}}

\usepackage{mathpazo}
\DeclareFontFamily{OT1}{eusb}{} \DeclareFontShape{OT1}{eusb}{m}{n} {<5> <6> <7> <8> <9> <10> <11> <12> <14.4> eusb10}{}
\DeclareMathAlphabet{\eusb}{OT1}{eusb}{m}{n}

\DeclareFontFamily{OT1}{eusm}{} \DeclareFontShape{OT1}{eusm}{m}{n} {<5> <6> <7> <8> <9> <10> <11> <12> <14.4> eusm10}{}
\DeclareMathAlphabet{\eusm}{OT1}{eusm}{m}{n}

\DeclareFontFamily{OT1}{eufm}{} \DeclareFontShape{OT1}{eufm}{m}{n} {<5> <6> <7> <8> <9> <10> <11> <12> <14.4> eufm10}{}
\DeclareMathAlphabet{\mathfrak}{OT1}{eufm}{m}{n}

\DeclareFontFamily{OT1}{fraktura}{}
\DeclareFontShape{OT1}{fraktura}{m}{n} {<5> <6> <7> <8> <9> <10> <11> <12> <13> <14.4> [1.1] eufm10}{}
\DeclareMathAlphabet{\fraktura}{OT1}{fraktura}{m}{n}

\DeclareFontFamily{OT1}{cmfi}{} \DeclareFontShape{OT1}{cmfi}{m}{n} {<5> <6> <7> <8> <9> <10> <11> <12> <13> <14.4> [0.9] cmfi10}{}
\DeclareMathAlphabet{\cmfi}{OT1}{cmfi}{b}{n}

\DeclareFontFamily{OT1}{cmss}{} \DeclareFontShape{OT1}{cmss}{m}{n} {<5> <6> <7> <8> <9> <10> <11> <12> <13> <14.4> cmss10}{}
\DeclareMathAlphabet{\cmss}{OT1}{cmss}{m}{n}
\newtheoremstyle{thm}{1.5ex}{1.5ex}{\itshape\rmfamily}{} {\bfseries\rmfamily}{}{2ex}{}

\newtheoremstyle{def}{1.5ex}{1.5ex}{\rmfamily\sl}{} {\bfseries\rmfamily}{}{2ex}{}

\newtheoremstyle{rem}{1.3ex}{1.3ex}{\rmfamily}{} {\bfseries\rmfamily}{}{2ex}{}

\newtheoremstyle{ass}{1.5ex}{1.5ex}{\rmfamily}{} {\bfseries\rmfamily}{}{2ex}{}

\newenvironment{proofsect}[1] {\vskip0.1cm\noindent{\rmfamily\itshape#1.}}{\qed\vspace{0.15cm}}

\theoremstyle{thm}
\newtheorem{theorem}{Theorem}[section]
\newtheorem{lemma}[theorem]{Lemma}
\newtheorem{proposition}[theorem]{Proposition}

\newtheorem*{Main Theorem}{Main Theorem.}
\newtheorem{corollary}[theorem]{Corollary}

\newtheorem{problem}[theorem]{Problem}

\newtheoremstyle{named}{}{}{\itshape}{}{\bfseries}{}{.5em}{\thmnote{#3}}
\theoremstyle{named}

\theoremstyle{def}
\newtheorem{definition}[theorem]{Definition}

\theoremstyle{ass}
\newtheorem{assumption}[theorem]{Assumption}

\theoremstyle{rem}
\newtheorem{remark}[theorem]{{Remark}}

\numberwithin{equation}{section}

\renewcommand{\section}{\secdef\sct\sect}
\newcommand{\sct}[2][default]{\nopagebreak\refstepcounter{section}
\addcontentsline{toc}{section}
{{\tocsection {}{\thesection}{\!\!\!\!#1\dotfill}}{}}
\vspace{0.7cm}
\nopagebreak
\centerline{ %\large
\scshape\arabic{section}.\ #1} \nopagebreak\vspace{0.2cm}\nopagebreak}
\newcommand{\sect}[1]{
\vspace{0.4cm} \centerline{\large\scshape\rmfamily #1}
\vspace{0.2cm}}

\renewcommand{\subsection}{\secdef\subsct\sbsect}
\newcommand{\subsct}[2][default]{\refstepcounter{subsection}
\addcontentsline{toc}{subsection}
{{\tocsection{\!\!}{\hspace{1.2em}\thesubsection}{\!\!\!\!#1\dotfill}}{}}
\nopagebreak\vspace{0.45\baselineskip} {\flushleft\bf
\arabic{section}.\arabic{subsection}~\bf #1.~}
\\*[3mm]\noindent
\nopagebreak}
\newcommand{\sbsect}[1]{\vspace{0.1cm}\noindent
\textbf{#1.~}\vspace{0.1cm}}

\renewcommand{\subsubsection}{
\secdef \subsubsect\sbsbsect}
\newcommand{\subsubsect}[2][default]{
\refstepcounter{subsubsection}
\addcontentsline{toc}{subsubsection}{{\tocsection{\!\!}
{\hspace{3.05em}\thesubsubsection}{\!\!\!\!#1\dotfill}}{}}
\nopagebreak
\vspace{0.15\baselineskip} \nopagebreak {\flushleft\rmfamily
\itshape\arabic{section}.\arabic{subsection}.\arabic{subsubsection}
\ \rmfamily #1\/.}\ }
\newcommand{\sbsbsect}[1]{\vspace{0.1cm}\noindent
\rmfamily \itshape
\arabic{section}.\arabic{subsection}.\arabic{subsubsection} \
\sffamily #1\/.\ }

\renewcommand{\caption}[1]{%
\vglue0.5cm
\refstepcounter{figure}
\begin{minipage}{0.9\textwidth}\small {\sc Figure~\thefigure. }#1\end{minipage}}

\newcommand{\supp}{\operatorname{supp}}

\newcommand{\textd}{\text{\rm d}\mkern0.5mu}

\newcommand{\texte}{\text{\rm  e}\mkern0.7mu}

\newcommand{\Cov}{\text{\rm \Cov}}

\newcommand{\CC}{\mathcal C}
\newcommand{\DD}{\mathcal D}
\newcommand{\EE}{\mathcal E}
\newcommand{\FF}{\mathcal F}

\newcommand{\E}{\mathbb E}

\newcommand{\BbbP}{\mathbb P}
\newcommand{\Q}{\mathbb Q}
\newcommand{\R}{\mathbb R}

\newcommand{\Z}{\mathbb Z}

\newcommand{\twoeqref}[2]{(\ref{#1}--\ref{#2})}

\newcommand{\cc}{{\text{\rm c}}}
\newcommand{\hate}{\hat{\text{\rm e}}}

\newcommand{\w}{\omega}

\newcommand{\IA}{\mathds{1}}

\newcommand{\Fs}{\mathscr{F}}

\def\myffrac#1#2 in #3{\raise 2.6pt\hbox{$#3 #1$}\mkern-1.5mu\raise 0.8pt\hbox{$#3/$}\mkern-1.1mu\lower 1.5pt\hbox{$#3 #2$}}

\newcommand\frakp{\mathfrak p}

\newcommand{\wt}{\widetilde}

\newcommand{\laweq}{\,\overset{\text{\rm law}}=\,}

\newcommand{\hm}{\mu}%{{\hat{\mu}}}
\newcommand{\hn}{\nu}%{{\hat{\nu}}}
\newcommand{\tn}{\nu}
\newcommand{\tC}{C}

\begin{document}
\allowdisplaybreaks

\vglue-5mm
\title[Long-range random conductance models \hfill \version\hfill]
{\large Quenched Invariance Principle for a class of random conductance models with long-range jumps}

\author[\hfill  \version \hfill Biskup, Chen, Kumagai and Wang]
{Marek~Biskup$^1$, Xin Chen$^2$, Takashi Kumagai$^3$ and Jian Wang$^4$}
\thanks{\hglue-4.5mm\fontsize{9.6}{9.6}
\selectfont\copyright\,\textrm{2021}\ \ \textrm{M.~Biskup, X.~Chen, T.~Kumagai and J. Wang.
Reproduction, by any means, of the entire
article for non-commercial purposes is permitted without charge.\vspace{2mm}}}
\maketitle

\vspace{-5mm}
\centerline{\textit{
$^1$Department of Mathematics, UCLA, Los Angeles, California, USA}}
\centerline{\textit{$^2$Department of Mathematics, Shanghai Jiao Tong University, Shanghai, P.R. China}}
\centerline{\textit{
$^3$Research Institute in Mathematical Sciences, Kyoto University, Kyoto, Japan}}
\centerline{\textit{$^4$College of Mathematics and Informatics \& Fujian Key Laboratory of
Mathematical}}
\centerline{\textit{
Analysis and Applications \& Center for Applied Mathematics of Fujian Province (FJNU),}}
\centerline{\textit{Fujian Normal University, Fuzhou, P.R. China}}

\vskip0.5cm
\begin{quote}
\footnotesize \textbf{Abstract:}
We study random walks on~$\Z^d$  (with $d\ge 2$) among stationary ergodic   random conductances $\{C_{x,y}\colon x,y\in\Z^d\}$ that  permit jumps of arbitrary length.
Our focus is on the Quenched Invariance Principle (QIP) which we establish by a combination of corrector methods, functional inequalities and  heat-kernel technology assuming that the $p$-th moment of
$\sum_{x\in\Z^d}C_{0,x}|x|^2$ and $q$-th moment of $1/C_{0,x}$ for~$x$ neighboring the origin are finite for some $p,q\ge1$ with $p^{-1}+q^{-1}<2/d$.  In particular, a QIP thus holds for random walks on long-range percolation graphs with connectivity exponents larger than~$2d$
in all $d\ge2$, provided all the nearest-neighbor edges are present.
  Although still limited by moment conditions, our method of proof is novel in that it avoids proving everywhere-sublinearity of the corrector. This is relevant because we  show that, for long-range percolation with exponents between $d+2$ and~$2d$, the corrector exists but fails to be sublinear everywhere. Similar examples are
constructed also for nearest-neighbor, ergodic conductances in  $d\ge3$  under the conditions complementary to those of the recent work of P. Bella and M. Sch\"affner \cite{BelSch}.  These examples elucidate the limitations of elliptic-regularity techniques that underlie much of the recent progress on these problems.
\end{quote}

\section{Introduction}
\noindent
Random walks among random conductances have seen much interest in recent years. The term ``random walk'' actually refers to a Markov chain whose states will be confined, for the purpose of the present paper, to the $d$-dimensional hypercubic lattice~$\Z^d$ and the transition probabilities $\cmss P(x,y)$ 
determined by a collection $\{C_{x,y}\colon x,y\in\Z^d\}$ of non-negative numbers via
\begin{equation}
\label{E:1.1}
\cmss P(x,y):=\frac{C_{x,y}}{\pi(x)},\quad\text{where}\quad\pi(x):=\sum_{y\in\Z^d}C_{x,y},
\end{equation}
where $\pi(x)\in(0,\infty)$ is assumed for all~$x\in\Z^d$.
The symmetry condition
\begin{equation}
C_{x,y}=C_{y,x},\qquad x,y\in\Z^d,
\end{equation}
is imposed and the common value is
called the \emph{conductance} of unordered edge $\langle x,y\rangle$. As is easily checked, $\pi$ is then a reversible measure for the chain. The setting naturally includes the cases when only nearest-neighbor jumps occur, i.e., those for which $C_{x,y}:=0$ whenever~$x$ and~$y$ are not nearest neighbors in~$\Z^d$ (this includes $x=y$).

Many ``ordinary'' random walks are naturally covered by the above setting; notably, the simple random walk when $C_{x,y}$ is set to one for nearest neighbors $x$ and~$y$ and zero otherwise, or random walks with $\alpha$-stable tail when $C_{x,y}:=|x-y|^{-(d+\alpha)}$, where $\alpha\in (0,2)$ and~$|x|$ denotes the Euclidean norm of~$x$. Our interest here is in the situation when $\{C_{x,y}\colon x,y\in\Z^d\}$ is itself random. Writing~$\BbbP$ for the law of the conductances and~$\E$ for its expectation, we impose:

\begin{assumption}
\label{ass1}
Throughout we assume:
\begin{enumerate}
\item[(1)] $\BbbP$ is stationary and jointly ergodic with respect to the shifts of~$\Z^d$, and
\item[(2)] denoting the origin of~$\Z^d$ by ``$0$'', we have $\E\,C_{0,0}<\infty$,
\end{enumerate}
\end{assumption}

\noindent
where we allow $C_{0,0}>0$ to permit the walk to ``hold'' in place but require finite mean to ensure that holding does not dominate its behavior.
In this framework we then ask what conditions guarantee various properties known for the ``ordinary'' random walks with symmetric jumps, e.g., lack of speed, recurrence/transience, etc. Here we will focus on the validity of an \emph{Invariance Principle}, i.e., convergence of the path law to Brownian motion under the diffusive scaling of space and time.

The so called \emph{Annealed} (or \emph{Averaged}) Invariance Principle (AIP) has been known since late 1980s (Kipnis and Varadhan~\cite{Kipnis-Varadhan}, De Masi, Ferrari, Goldstein and Wick~\cite{DMFGW1,DMFGW2}). The adjective ``annealed'' refers to the convergence taking place for a \emph{joint} law of the chain and the environment.
With Assumption~\ref{ass1} in force, this convergence was shown under the moment conditions
\begin{equation}
\label{E:1.3}
\E\biggl(\,\sum_{x\in\Z^d}C_{0,x}|x|^2\biggr)<\infty\quad\text{and}\quad\E\Bigl(\frac1{C_{0,x}}\Bigr)<\infty \text{ whenever }|x|=1.
\end{equation}
These are directly linked to the limiting Brownian motion having finite and positive variance and so, in this sense, can be regarded as optimal.

Much effort in the past 15 years went to derivations of \emph{Individual} or \emph{Quenched} Invariance Principle (QIP) where the convergence to Brownian motion takes place for a.e.\ sample of the random conductances. The influential initial study by Sidoravicius and Sznitman~\cite{Sidoravicius-Sznitman}, where a QIP was proved for all uniformly-elliptic nearest-neighbor conductances, elucidated the need for additional ingredients compared to AIP; namely, estimates on the heat kernel. Analyses of the simple random walk on supercritical percolation clusters (Sidoravicius and Sznitman~\cite{Sidoravicius-Sznitman}, Berger and Biskup~\cite{Berger-Biskup}, Mathieu and Piatnitski~\cite{Mathieu-Piatnitski}) then paved the way to a complete resolution of all i.i.d.\ nearest-neighbor random conductance models (Mathieu~\cite{Mathieu-CLT}, Biskup and Prescott~\cite{Biskup-Prescott}, Barlow and Deuschel~\cite{Barlow-Deuschel} and Andres, Barlow, Deuschel and Hambly~\cite{ABDH}).

Compared to the i.i.d.\ cases, our understanding of general non-uniformly elliptic conductances remains only partial and often restricted to special cases. For nearest-neighbor models satisfying Assumption~\ref{ass1}, the restriction may come as a limitation on the spatial dimension. %%:
%% We change : by . here.
Indeed, as shown in
Biskup \cite[Exercise 4.4 and Theorem 4.7]{Biskup-review},  a QIP holds true in $d=1,2$
whenever
\begin{equation}
\label{E:1.5}
|x-y|=1\quad\Rightarrow\quad\E C_{x,y}<\infty\quad\text{and}\quad \E\Bigl(\frac1{C_{0,x}}\Bigr)<\infty.
\end{equation}
These are deemed sharp in light of \eqref{E:1.3} although examples violating \eqref{E:1.5} exist for which QIP fails yet AIP holds (Barlow, Burdzy and Tim\'ar~\cite{BBT}).
Another way to limit the form of the distribution is through decay of correlations. Indeed, Procaccia, Rosenthal and Sapozhnikov~\cite{PRS} proved a QIP in correlated percolation models subject to technical conditions on  correlation~\hbox{decay}.

A third type of restriction comes via moment conditions on individual (nearest-neigh\-bor) conductances. These can be expressed by means of numbers $p,q\ge1$ such that
\begin{equation}
\label{E:1.6qq}
|x-y|=1\quad\Rightarrow\quad
C_{x,y}\in L^p(\BbbP)\quad\text{and}\quad \frac1{C_{x,y}}\in L^q(\BbbP).
\end{equation}
For these Andres, Deuschel and Slowik~\cite{ADS}  proved  a QIP under the condition  $1/p+1/q<2/d$. This extends to a local QIP~\cite{ADS2} and
also to the control of the heat kernel~\cite{ADS3}.
Bella and Sch\"affner~\cite{BelSch}  recently improved the methods of~\cite{ADS} and gave a proof of a QIP under a slightly weaker condition
\begin{equation}
\label{E:1.7**}
\frac1p+\frac1q<\frac2 {d-1}.
\end{equation}
 The key {additional} ingredient for the improvement of the moment condition from \eqref{E:1.6qq} into \eqref{E:1.7**} is
a local boundedness result for finite difference
equations in divergence forms
under essentially optimal moment conditions; see \cite[Theorem 2]{BelSch}.
We will show that  \eqref{E:1.7**}  is, in fact, infinitesimally close to sharp, at least for the method of proof.

\newcommand{\eff}{\text{\rm eff}}

The main goal of the present paper is to push the control of a QIP to include models with arbitrarily large jumps. We will work under the following moment assumption:

\begin{assumption}\label{Asmp2}\it
Assume $d\ge2$ and that there are $p,q\in(1,\infty)$ satisfying
\begin{equation}
\label{E:1.7*}
\frac1p+\frac1q<\frac2d
\end{equation}
such that
\begin{equation}
\label{E:1.8*}
\sum_{x\in\Z^d}C_{0,x}|x|^2\in L^p(\BbbP)
\end{equation}
and
\begin{equation}
\label{E:2.1}
\frac1{C_{0,x}}\in L^q(\BbbP)\text{\rm\ whenever }|x|=1.
\end{equation}
In particular, $C_{0,x}>0$ for all~$|x|=1$ $\BbbP$-a.s.
\end{assumption}

\noindent
Assumption~\ref{Asmp2} is a direct extension of the conditions from the work \cite{ADS}, which is thus subsumed by the present paper (albeit with different proofs). We note that $C_{x,y}>0$ for all nearest-neighbors $x$ and~$y$ ensures that the underlying Markov chain is irreducible.

\section{Main results}
\label{sec2}\noindent
We will invariably work with random collections of conductances $\{C_{x,y}=C_{y,x}\colon x,y\in\Z^d\}$ such that $\pi(x)\in(0,\infty)$ for a.e.\ sample from~$\BbbP$. This ensures that the transition
probability in \eqref{E:1.1} is well defined almost surely in all cases of interest. We will write~$Z:=\{Z_n\colon n\ge0\}$ for the paths of the associated discrete-time Markov chain and use~$P^x$ to denote its law subject to the initial condition~$P^x(Z_0=x)=1$.

\subsection{QIP for general conductances}
As noted above, our main point of interest is the validity of the Quenched Invariance Principle --- or QIP for short --- which we formalize as follows:

\begin{definition}
Let $\CC([0,T])$ denote the space of continuous functions on~$[0,T]$ endowed with the supremum topology. Given a path $\{Z_n\colon n\ge0\}$ of the chain, define
\begin{equation}
\label{E:2.1a}
B^{(n)}(t):=\frac1{\sqrt n}\Bigl(Z_{\lfloor tn\rfloor}+(tn-\lfloor tn\rfloor)(Z_{\lfloor tn\rfloor+1}-Z_{\lfloor tn\rfloor})\Bigr),\qquad t\ge0.
\end{equation}
We will say that a QIP holds if for each~$T>0$ and $\BbbP$-a.e.~realization of the conductances, the law of $B^{(n)}$ induced on $\CC([0,T])$ by~$P^0$ tends weakly, as $n\to\infty$, to that of a Brownian motion whose covariance is non-degenerate and constant~a.s.
\end{definition}

Our main result is then:

\begin{theorem}
\label{thm:1}
Suppose~$d\ge2$. Then
a QIP holds under Assumptions~$\ref{ass1}$ and $\ref{Asmp2}$.
\end{theorem}

As noted earlier, for nearest neighbor conductances that are bounded from below, our results degenerate to those of~\cite{ADS}. An interesting corollary arises in the context of random walks on a family of long-range percolation graphs. These graphs are obtained from a ``nice'' underlying graph, in our case~$\Z^d$, by adding edges independently with probability that depends only on the displacement between the endpoints. While this probability is typically assumed to decay as a power of the distance, our formulation only requires a summability condition.

\begin{corollary}
\label{cor-2.3}
Let~$d\ge2$.
Given a function $\frakp\colon\Z^d\to[0,1]$ such that
\begin{enumerate}
\item[(1)] $\frakp(x)=\frakp(-x)$ for all~$x\in\Z^d$,
\item[(2)] $\frakp(0)=0$ and $\frakp(x)=1$ whenever~$|x|=1$,
\item[(3)] $\sum_{x\in\Z^d}\frakp(x)
|x|^{2p}<\infty$ for some $p>d/2$,
\end{enumerate}
consider a random graph with vertices~$\Z^d$ and an (unoriented) edge between~$x$ and~$y$ present with probability~$\frakp(y-x)$, independently of all other edges. Then a QIP holds for the simple random walk on this graph.
\end{corollary}

 To make the setting clear, the conductances in Corollary~\ref{cor-2.3} take only values in~$\{0,1\}$  with $\BbbP(C_{x,y}=1)=\frakp(x-y)$. In each step, the ``simple random walk'' on the graph selects an available neighbor at random. This is meaningful because the  graphs in Corollary~\ref{cor-2.3} are automatically connected (since $\frakp(x)=1$ for $|x|=1$) and of finite degree at every vertex (as ensured by $\sum_{x\in\Z^d}\frakp(x)<\infty$). Restricting attention to power-law decaying connection probabilities, our results show that if
\begin{equation}
\label{E:2.7}
\frakp(x)=|x|^{-s+o(1)},\qquad |x|\to\infty,
\end{equation}
for some $s>2d$, then a QIP holds in $d\ge2$. This is sharp in $d=2$ but weaker than expected in~$d\ge3$ because, based on \eqref{E:1.3}, we expect a QIP to hold for all $s>d+2$.
Note also that, 
since $C_{xy}\in\{0,1\}$, 
\begin{equation}
\E\biggl(\Bigl(\sum_{x\in\Z^d}C_{0,x}|x|^2\Bigr)^p\biggr)\ge
\sum_{x\in\Z^d}\frakp(x)
|x|^{2p}
\end{equation}
whenever~$p\ge1$. Condition~(3) is thus necessary for Assumption~\ref{Asmp2} to hold.

In the regime $s\in(d,d+2)$, the simple random walk 
on the long-range percolation graph is supposed to scale to a stable process with index~$\alpha:=s-d$. This was proved for $\alpha\in(0,1)$  in all $d\ge1$ by Crawford and Sly~\cite{Crawford-Sly1,Crawford-Sly2} under the $L^r$-space topology for any $r\in[1,\infty)$ (which is weaker than the Skorohod topology).  In $d=1$ the regime when a QIP holds extends to all $\alpha>1$, i.e., even beyond the summability of $\sum_{x\in\Z^d}|x|^{2}\frakp(x)$,
cf. \cite[Theorem~1.2]{Crawford-Sly2} (see also Kumagai and Misumi~\cite[Theorem 2.2]{KumMi} concerning heat kernels).
This is due to absence of percolation and the existence of cut-points. A corrector-based approach exists as well (Zhang and Zhang~\cite{Zhang^2}).

We note that, in the regime $s\in (d,2d)$, the long-range percolation graph is rather different from~$\Z^d$. Indeed, as shown by Biskup~\cite{B04,B11} and Biskup and Lin~\cite{BL}, the graph distances grow polylogarithmically with the Euclidean distance and balls in the intrinsic metric thus exhibit stretched-exponential volume growth. When $s=2d$, the scaling of intrinsic metric relative to Euclidean one is only polynomial, but with exponents strictly less than one (Ding and Sly~\cite{DS}). Notwithstanding, for $s>d+2$, these do not seem to affect the asymptotic of the random walk, at least at the level of AIP.

\subsection{Lack of everywhere sublinearity}
\noindent
The second set of our results address limitations of the techniques presently used for proofs of the QIP. This requires introducing
the basic object of stochastic homogenization, the so-called \emph{corrector}~$\chi$. Consider the generator $\cmss L:=\cmss P-\text{id}$ associated with the discrete-time Markov kernel~$\cmss P$. Explicitly, $\cmss L$ acts on finitely-supported $f\colon\Z^d\to\R$ as
\begin{equation}
\label{E:2.4}
(\cmss L f)(x):=\frac{1}{\pi(x)}\sum_{y\in\Z^d}C_{x,y}\bigl[f(y)-f(x)\bigr].
\end{equation}
Under Assumption~\ref{ass1}, the conditions \eqref{E:1.3} permit the construction of
 a random function $\chi\colon\Z^d\to\R^d$ which is characterized by the following properties:
\begin{enumerate}
\item[(1)] normalization $\chi(0)=0$,
\item[(2)] stationarity of increments under the shifts of~$\Z^d$
\begin{equation}
\bigl\{\chi(x)-\chi(y)\colon x,y\in\Z^d\bigr\} \,\,\,\overset{\text{law}}=\,\,\, \bigl\{\chi(x-y)\colon x,y\in\Z^d\bigr\},
\end{equation}
\item[(3)] weighted square-integrability $\E(\sum_{x\in\Z^d}C_{0,x}|\chi(x)|^2)<\infty$,
\item[(4)] harmonicity of the function
\begin{equation}
\label{E:2.5}
\Psi(x):=x+\chi(x)
\end{equation}
in the sense that
\begin{equation}
\label{E:2.6}
(\cmss L\Psi)(x)=0,\qquad x\in\Z^d.
\end{equation}
(For this reason, $\Psi$ is sometimes referred to as ``harmonic coordinate.'')
\end{enumerate}
We refer to, e.g., Biskup~\cite[Proposition~3.7]{Biskup-review} for a detailed exposition and proofs of this otherwise completely classical material.

\begin{remark}
In all QIPs discussed in this paper, the covariance matrix $\Sigma=(\Sigma_{ij})$ of the limiting Brownian motion is related to the corrector via
\begin{equation}
\label{E:sigma}
\Sigma_{ij}=\frac1{\E\pi(0)}\E\Bigl(\sum_{x\in\Z^d}C_{0,x}\bigl(x_i+\hate_i\cdot\chi(x)\bigr)\bigl(x_j+\hate_j\cdot\chi(x)\bigr)\Bigr),
\end{equation}
where $x_i$ denotes the $i$-th Cartesian component of~$x$ and $\hate_i$ denotes the unit vector in the~$i$-th coordinate direction.
Note that $\Sigma$ is non-degenerate and finite under \eqref{E:1.3}; see Proposition~\ref{prop-FDD} for an explicit statement in this vain. 
%%To Marek: What is the meaning of "in this vain"? 
\end{remark}

It it well known (see~\cite[Lemma~4.8]{Biskup-review}) that \eqref{E:1.3} ensures that $\chi$ is sublinear along coordinate directions in the sense that $\frac1n\chi(nx)\to0$ $\BbbP$-a.s.\ as~$n\to\infty$ for each~$x\in\Z^d$. This is what gives a QIP in all $d=1$ situations. In $d\ge2$ this implies \emph{sublinearity on average},
\begin{equation}
\label{E:sub}
\forall\delta>0\colon\quad
\lim_{n\to\infty}\,\frac1{n^d}\sum_{|x|\le n}\IA_{\{|\chi(x)|>\delta n\}}=0,\quad\BbbP\text{-a.s.}
\end{equation}
see~\cite[Proposition~4.15]{Biskup-review}.
A key step underlying all of the aforementioned approaches to QIP is the proof of \emph{sublinearity everywhere},
\begin{equation}
\label{E:1.9}
\lim_{n\to\infty}\,\frac1n\,\max_{x\colon|x|\le n}\bigl|\chi(x)\bigr|=0,\qquad \BbbP\text{-a.s.}
\end{equation}
This is known to be sufficient to get a Brownian limit for diffusively-scaled paths of the Markov chain (see, e.g., Biskup~\cite[Section~4.2]{Biskup-review}, Kumagai~\cite[Section~8.4]{Kumagai-review} or \cite{ADS}).

While our proof of Theorem~\ref{thm:1} avoids everywhere sublinearity, the conditions we work under are still generally necessary for everywhere sublinearity to hold:

\begin{theorem}
\label{thm-2.6}
Let $d\ge3$ and consider the long-range percolation graph obtained from~$\Z^d$ as above with $\frakp$ having the asymptotic \eqref{E:2.7} for some $s\in(d+2,2d)$. Then the corrector is (well defined yet) not sublinear everywhere. \end{theorem}

\noindent
Note that this is true despite the conjecture that a QIP holds for all~$s>d+2$.
A natural question is whether such examples can be constructed  also for nearest-neighbor conductances. This is answered in:

\begin{theorem}
\label{thm-2.8}
Suppose $d\ge3$ and let~$p,q\ge1$ be such that
\begin{equation}
\label{E:2.8}
\frac1p+\frac1q>\frac 2{d-1}.
\end{equation}
Then there is a law~$\BbbP$ on nearest-neighbor conductances satisfying Assumption~$\ref{ass1}(1)$ and
\begin{equation}
\label{E:2.9a}
|x|=1\quad\Rightarrow\quad C_{0,x}\in L^p(\BbbP)\quad\text{and}\quad \frac1{C_{0,x}}\in L^q(\BbbP)
\end{equation}
for which the corrector is (well defined yet) not sublinear everywhere.
\end{theorem}
Modulo a boundary case, condition \eqref{E:2.8} is
complementary to \eqref{E:1.7**} under which
Bella and Sch\"affner \cite{BelSch} proved that the corrector is sublinear everywhere and thus a QIP holds. While Theorem~\ref{thm-2.8} presents counterexamples only to the method of proof, rather than the QIP itself, it makes it unlikely that the elliptic-regularity methods underlying \cite{ADS,BelSch} would ever yield a proof of a QIP  in nearest-neighbor conductance models under the presumably optimal conditions \eqref{E:1.5}.  It appears that more promising strategies are to either try to infer a QIP directly from the AIP (which is known to hold under \eqref{E:1.5}) or restrict proofs of sublinearity of the corrector to just the set of vertices visited by the random walk. An attempt in the latter direction has been made by Ba and Mathieu~\cite{Ba-Mathieu}, albeit for diffusions in periodic environments.

\subsection{Main ideas}
\noindent
 Although our proofs are based on a combination of the corrector method with heat-kernel technology, our strategy is somewhat different from that used in proofs of QIPs so far. Through the use of functional inequalities we first control the first exit times of the walk from large balls. These are used to prove tightness of diffusively-scaled Markov-chain paths. The proof of a QIP then boils down to the proof of a quenched Central Limit Theorem (CLT). For this we use the corrector method but, since this is ``just'' a CLT, with everywhere sublinearity replaced by sublinearity on average.

As usual, we work primarily with continuous-time versions of our random walk. A key innovation is the use of
\begin{equation}
\label{E:3.1}
\hn(x):=\sum_{y\in\Z^d}C_{x,y}|x-y|^2
\end{equation}
as the time-change measure for the walk. The need for this particular normalization was discovered in the derivation of off-diagonal heat-kernel estimates using the so called Davies method; cf the proof of Proposition~\ref{thm:hitest}. Another instance where this measure naturally appears are estimates on Dirichlet forms of spatially-mollified functions; see the proof of Proposition~\ref{P:Sob}.  Notwithstanding, the use of the time change by~$\hn$ is purely technical. In particular, a QIP will hold for all time-change measures that have finite expectation under the invariant law from the point of the particle.

Unlike the recent work~\cite{ADS,ADS2} on QIPs under moment conditions, in order to control the heat kernel and exit probabilities we do not use complicated inductive schemes such as Moser or De Georgi iterations. Instead, we base our argument on \emph{localization}, which amounts to restricting jumps larger than unity to only a finite ``active'' ball, and \emph{truncation}, by which we discard jumps larger than a constant (called~$\kappa$ below) multiple of the ``active'' ball radius. The localization helps us control ``small'' jumps using Sobolev inequalities and standard techniques from heat kernel theory. The contribution of  ``large'' jumps is managed with the help of so called Meyer's construction  (see~\cite{Mey75, BGK}) which, to put simply, is a way to control the transition probabilities of a Markov process with jumps by explicating, just as in the integral form of the Kolmogorov backward equation, the first ``large'' jump of the process.

While the idea to combine localization argument with Meyer's construction is drawn from
a recent paper by some of the authors \cite{CKW} (see Section 2.2 therein for more details), their extension to the present context requires non-trivial adaptations and generalizations.
A key point for us is that localization enables
a class of scale dependent Sobolev inequalities and Davies's method that yield estimates  for the heat kernel and exit times (see Propositions~\ref{P:Sob} and~\ref{thm:hitest} below).
With these in hand, a QIP can be proved from sublinearity on average of the associated corrector alone.

Although we do not address everywhere sublinearity of the corrector in our parameter regime, we suspect that it does hold under Assumption~\ref{ass1} and \ref{Asmp2}. The counterexamples in the nearest-neighbor case are strongly inspired by analogous examples of i.i.d.\ nearest neighbor conductances (Mathieu and Remy~\cite{MR}, Berger, Biskup, Hoffman and Kozma~\cite{BBHK}, Biskup and Boukhadra~\cite{BB}) for which the return probabilities exhibit strongly subdiffusive decay while the path distribution still scales to a non-degenerate Brownian motion. The key mechanism there is trapping.

\subsection{Open problems}
We finish by stating a few open problems that naturally build on the results of the present note.
As a starter, we pose:

\begin{problem}\label{p1}
Under the conditions of Theorem~\ref{thm:1}, prove a local CLT by showing, e.g.,
\begin{equation}
\lim_{n\to\infty}
\sup_{\begin{subarray}{c}
x\in(2\Z)^d\\|x|\le\sqrt n
\end{subarray}}
\Bigl| (2n)^{d/2}\cmss P^{2n}(0,x)-k_1(x/\sqrt n)\Bigr|=0,\quad \BbbP\text{\rm-a.s.},
\end{equation}
where $k_1$ is the probability density of a centered normal with covariance \eqref{E:sigma}.
\end{problem}

\noindent
We believe that this holds under the same conditions as a CLT by analogy with the nearest-neighbor situations in~\cite{ADS,ADS2}. For nearest-neighbor
variable speed random walks in ergodic environments, a local CLT is  
obtained in \cite{ADS2, AT, BS3} 
under some moment condition whereas a QIP
has been established in \cite{ABDH,ADS,BBT,Barlow-Deuschel,BelSch}.
In particular, according to \cite[Proposition 1.5]{DF}, when $d\ge 4$ 
and $q=\infty$, the moment condition \eqref{E:1.7**}
(note that the dimension~$d$ of \cite{DF} corresponds to $d-1$ in the present paper) is 
sufficient and, except for the equality, generally necessary for   a local CLT   to hold for variable speed random walks in ergodic environments.

 Techniques to prove quenched local CLT results are definitely available. Indeed, Andres, Chiarini and Slowik~\cite{ACS} extended the iteration methods underlying~\cite{ADS,ADS2} to non-elliptic situations under a condition slightly weaker than \eqref{E:1.7*}. Very recently, Bella and Sch\"affner~\cite{BS2,BS3} use the regularity theory of weak solutions of elliptic
equations to prove a local CLT for nearest-neighbor
variable speed random walks even under the weaker moment condition \eqref{E:1.7**}. Still, a quenched local CLT in both papers \cite{ADS2,BS3} is based on parabolic Harnack inequalities. The problem is that, for models with arbitrarily large jumps,
elliptic Harnack inequalities may not hold in general even for large balls and parabolic Harnack inequalities may thus not hold either. See \cite[Section 4.2.2]{CKW2} for related discussions on random conductance models with stable-like jumps.

Another extension that we believe should be possible by a reasonably straightforward  adaptation  of the methods of the present work is the content of:

\begin{problem}
Let $d\ge2$ and suppose $\frakp\colon\Z^d\to[0,1]$ obeys \eqref{E:2.7} with $s>2d$ and $\frakp(x):=p$ when~$|x|=1$ for some $p\in[0,1)$. Assume that the random graph with vertex set~$\Z^d$ and an edge between~$x$ and~$y$ present with probability  $\frakp(y-x)$
independently of other edges, contains an infinite connected component $\CC^\infty$ a.s. Prove that the simple random walk on $\CC^\infty$ obeys a QIP.
\end{problem}

\noindent
Here the key challenge is the potential absence (as even $p=0$ is allowed) of nearest-neighbor edges in the computations involving Dirichlet forms in our proofs. Barlow~\cite{Barlow} and the recent work of Flegel, Heida and Slowik~\cite{FHS} provide good possible starting points.

In light of Theorems~\ref{thm-2.6} and~\ref{thm-2.8}, a different strategy than used so far is needed to get a QIP beyond the regime marked by \eqref{E:1.7**} or \eqref{E:1.8*} and, in particular, for long-range percolation graphs with decay exponents $s\in(d+2,2d]$.
Here we propose to start with:

\begin{problem}
Consider the long-range percolation graph with exponents $s\in(d+2,2d]$ and
$\frakp(x)=1$ for~$|x|=1$. Prove a QIP.
\end{problem}

\noindent
The requirement $\frakp(x)=1$ ensures that the underlying graph is connected. A key obstacle is thus the lack of the Sobolev inequalities underlying our proofs. Although the corrector fails to be everywhere sublinear in these cases, this is not an obstacle for our approach, for which sublinearity on average is sufficient.
We find it worthwhile to start by addressing the non-percolating regime, i.e., the situations when, upon removal of the nearest-neighbor edges, the graph does not contain an infinite connected component a.s.

 A considerably more robust way to go beyond the $p,q$-conditions would be to prove corrector sublinearity along typical paths of the Markov chain
\begin{equation}
\label{E:2.15o}
\max_{1\le k\le n}\frac1{\sqrt n}\bigl|\chi(Z_k)\bigr|\,\underset{n\to\infty}{\overset{P^0}\longrightarrow}\,0,\quad\text{$\BbbP$-a.s.}
\end{equation}
 This is, in fact, what underlies the known proofs of the AIP under the optimal conditions \eqref{E:1.3} and even the present paper goes part of the way along this line. Ba and Mathieu~\cite{Ba-Mathieu} have been able to utilize this strategy to prove a QIP for a continuum diffusion in a random environment subject to a periodicity requirement.

 The approach of \cite{Ba-Mathieu} is based on Dirichlet form theory, time change arguments and new weighted Sobolev-type inequalities for integrable potentials. A key benefit of  the periodicity of the environment is that only global weighted Sobolev-type inequalities
and on-diagonal heat-kernel upper bounds are required (while we have to work with scale-dependent Sobolev inequalities and control also off-diagonal heat kernel upper bounds).
A question relevant from the point of view of our counterexamples is whether the corrector in~\cite{Ba-Mathieu} is everywhere sublinear or not --- for if it is, then the periodicity assumption is perhaps too strong to tell us much in our context.

Our last question, which is undoubtedly the one most ambitious, concerns the random walk on one-di\-men\-sional long-range percolation graphs (i.e., the setting of Corollary~\ref{cor-2.3}, (1,2) and \eqref{E:2.7}). Indeed, as noted above, there we get $(s-1)$-stable process convergence when $s\in(1,2)$ and a Brownian limit when~$s>2$.

\begin{problem}
Prove (quenched or annealed) convergence for suitably scaled random walk on one-dimensional long-range percolation graphs for $s=2$.
\end{problem}

\noindent We conjecture that all the $\alpha$-stable limits with $\alpha\in(1,2)$ somehow appear for~$s=2$. If so, we would expect that the index of stability depends on the precise asymptotic of the connection probabilities; i.e., on $\beta:=\lim_{|x|\to\infty}|x|^2\frakp(x)$. Since the $1/r^2$-percolation model  is known to exhibit multiple  phase transitions (cf Aizenman and Newman~\cite{AN}, Imbrie and Newman~\cite{IN}), the dependence of~$\alpha$ on~$\beta$ may even undergo interesting phase transitions as well.

\section{Functional inequalities and heat-kernel estimates}
\label{sec3}\nopagebreak\noindent
We will now move to the exposition of the proofs. In this section, we develop the main technical ingredients underlying the proof of QIP in Theorem~\ref{thm:1}.
We start by introducing continuous-time versions of our discrete-time Markov chains.

\subsection{Continuous time processes}
\label{3-1}\noindent
Recall that $Z:=\{Z_n\colon n\ge0\}$ denotes the discrete-time process on $\Z^d$ with transition probabilities $\cmss P(x,y)$ and associated stationary measure~$\pi$ as defined in~\eqref{E:1.1}. We will consider two continuous-time variants of~$Z$. The first one is the canonical variable-speed  chain $X:=\{X_t: t\ge0\}$ --- the VSRW --- obtained from~$Z$ by taking jumps
at independent exponential times
whose parameter at~$x$ is~$\pi(x)$. The process~$X$ is then a continuous-time Markov chain on~$\Z^d$ with the generator
\begin{equation}
({\cmss L}_X f)(x):=\sum_{y\in \Z^d}C_{x,y}\bigl[f(y)-f(x)\bigr].
\end{equation}
The counting measure $\hm(x):=1$ on~$\Z^d$ is stationary and reversible for~$X$.
Hence, the Dirichlet form $(D,\Fs)$ associated with the process $X$ is given by
\begin{equation}
\begin{aligned}
D(f,f)&:=
\sum_{x,y\in
\Z^d}C_{x,y}\bigl[f(y)-f(x)\bigr]^2  ,\quad
f\in
\Fs,\\
\Fs&:=\bigl\{f \in \ell^2(\hm)\colon D(f,f)<\infty\bigr\}.
\end{aligned}\end{equation}
Here, for any $p\in[1,\infty)$ and any measure~$\lambda$ on~$\Z^d$, let $\ell^p(\lambda)$ denote the space of $p$-integrable functions $f\colon\Z^d\to\R$ and
 denote by $\|f\|_{\ell^p(\lambda)}$ the corresponding $\ell^p$-norm.

Our second, and more important, continuous-time chain $Y:=\{Y_t:t\ge0\}$ will be a time change of the process~$X$ defined as follows:
\begin{equation}
Y_t:=X_{A^{-1}_t},\quad\text{where}\quad A^{-1}_t:=\inf\{s\ge 0\colon A_s>t\}\quad\text{for}\quad A_t:=\int_0^t\hn({X_s})\,\textd s,
\end{equation}
with $\hn(x)$  as in \eqref{E:3.1}.
Then~$Y$ is a continuous-time Markov chain on $\Z^d$ with the generator
\begin{equation}
({\cmss L}_Y f)(x):=\frac1{\hn(x)}\sum_yC_{x,y}\bigl[f(y)-f(x)\bigr]
\end{equation}
and~$Y$ is thus reversible with respect to~$\hn$.
(Alternatively, $Y$ can be defined directly from~$Z$ and independent exponentials that at~$x$ have parameter $\pi(x)/\hn(x)$.)
In particular, the Dirichlet form $(\wt D,\wt\Fs)$ associated with the process $Y$ is given by
\begin{equation}\begin{aligned}
\wt D(f,f)&:=\sum_{x,y\in
\Z^d}\hn(x)\frac{ C_{x,y}}{\nu(x)}\bigl[f(y)-f(x)\bigr]^2=D(f,f) ,\quad
f\in
\wt\Fs,\\
\wt\Fs&:=\bigl\{f \in \ell^2(\hn)\colon \wt D(f,f)<\infty\bigr\}.
\end{aligned}\end{equation}
We will henceforth think of the chains $Z$, $X$ and~$Y$ as defined on the same probability space, and write~$P^x$ for the joint law of their paths where (each) chain is at~$x$ at time zero a.s. 
We will use~$E^x$ to denote expectation with respect to~$P^x$.

The random processes $X$, $Y$ and~$Z$ on~$\Z^d$ naturally induce corresponding random processes on the space of random environments, via the ``point of view of the particle.'' These are stationary and reversible with respect to the measures $\Q_X$, $\Q_Y$ and~$\Q_Z$, respectively, defined by
\begin{equation}
\Q_X(\textd\omega):=\BbbP(\textd\omega),\quad
\Q_Y(\textd\omega):=\frac{\hn(0)}{\E\hn(0)}\BbbP(\textd\omega),\quad
\Q_Z(\textd\omega):=\frac{\pi(0)}{\E\pi(0)}\BbbP(\textd\omega),
\end{equation}
where~$\omega$ denotes a generic element from the sample space carrying the conductance law~$\BbbP$.
Thanks to our assumptions, all three measures are mutually absolutely continuous with respect to~$\BbbP$.  Moreover, according to Assumption \ref{ass1} and the definitions \eqref{E:1.1} and \eqref{E:3.1}, both of the measures $\pi$ and $\hn$ satisfy that $\pi(x)=\pi(0)\circ\tau_x$ and $\hn(x)=\hn(0)\circ\tau_x$ for all $x\in \Z^d$, where  $\{\tau_x\}_{x\in \Z^d}$ are the shifts of~$\Z^d$.  This structure ensures absence of finite-time blow-ups:

\begin{lemma}\label{thm:conserv}
Suppose Assumptions~$\ref{ass1}$ and $\ref{Asmp2}$ hold. Then both $X$ and $Y$ are conservative under~$P^x$, for all~$x\in\Z^d$ and $\BbbP$-a.e.\ sample of the conductances.
\end{lemma}

\begin{proof}
We will invoke a standard criterion (see, e.g., Liggett~\cite[Chapter~2]{Liggett}) plus some stationarity and the fact that~$X$ and~$Y$ are derived from the discrete-time Markov chain~$Z$. Focusing on~$X$ first, we have $X_t=Z_{N_t}$ for $N_t:=\sup\{n\ge0\colon T_1+\dots+T_n\le t\}$ where, conditional on~$Z$, the random times~$\{T_k\colon k\ge1\}$ are independent exponentials with~$T_k$ having parameter~$\pi(Z_{k-1})$. Thanks to the 1st and 2nd Borel-Cantelli lemmas,
\begin{equation}
\sum_{k\ge1}T_k=\infty\quad\text{a.s.}\qquad\Leftrightarrow\qquad \sum_{k\ge1} E^x(T_k|Z)=\infty\quad\text{a.s.}
\end{equation}
so no blow-ups occur if and only if the sum on the right diverges a.s.
Now $E^x(T_k|Z)=1/\pi(Z_{k-1})$ and so we need $\sum_{k\ge0}1/\pi(Z_k)=\infty$ a.s.
The
stationarity and ergodicity of~$\Q_Z$ for the process on environments induced by~$Z$ imply
\begin{equation}
\frac1n\sum_{k=0}^{n-1}1/\pi(Z_k)\,\underset{n\to\infty}\longrightarrow\, \E_{\Q_Z}(1/\pi(0))=1/\E\pi(0)\quad\text{\ a.s.}
\end{equation}
The limit is positive since $\E\pi(0)<\infty$ by Assumption \ref{Asmp2}. In particular, we have $\sum_{k\ge0}1/\pi(Z_k)=\infty$ a.s.

 The argument for~$Y$ process is analogous; only that~$T_{k+1}$ is now (conditionally on~$Z$) exponential with parameter~$\pi(Z_k)/\hn(Z_k)$. Here we also need
$0<\E\hn(0)<\infty$
as implied by Assumption \ref{Asmp2}  as well as $\hn(x)=\hn(0)\circ\tau_x$ for all $x\in \Z^d$ as noted above.
\end{proof}

\subsection{Localization and truncation}
\label{3-2}\noindent
Our proof focuses on the process~$Y$. The main challenge is to control the contribution of large jumps. As noted earlier, we do this by way of localization, which is a change of the environment that limits all the complexity to a finite ball, and {truncation}, where we remove jumps larger than a certain cutoff from the environment. We remark that the idea of considering localized modifications of non-local Dirichlet forms has appeared in \cite[Section 2.2]{CKW}, but here the construction is more delicate as we need to modify both the conductances and the reference measure.

We start by localization.
Denote
\begin{equation}
B(x,R):=x+([-R,R]^d\cap\Z^d).
\end{equation}
 For any integer $R\ge1$, let
\begin{equation}
\label{E:3.11eq}
\tC_{x,y}^R:=
\begin{cases}
C_{x,y}, &\text{if}\ x\in B(0,2R)\ \text{or}\ y\in B(0,2R),\\
 1,& \text{if}\  x\notin B(0,2R)\ \text{and}\ y\notin B(0,2R) \text{ and } |x-y|=1,\\
 0,& \text{otherwise}
\end{cases}
\end{equation} and
\begin{equation}
\label{E:3.12eq}
\tn^R(x):=
\begin{cases}
\hn(x),&\text{if}\ x\in B(0, 2R),\\
1+\hn(x),&\text{if}\ x\in B(0, 4R)\smallsetminus B(0,2R),\\
 1,&\text{if}\ x\notin B(0,4R)
\end{cases}
\end{equation}
and define a symmetric regular Dirichlet form $(\wt D^R,\wt\Fs^R)$ by
\begin{equation}\begin{aligned}
\wt D^R(f,f)&:=\sum_{x,y\in
\Z^d}\tC_{x,y}^R\bigl[f(y)-f(x)\bigr]^2,\quad
f\in
\wt\Fs^R,\\
\wt\Fs^R&:=\bigl\{f \in \ell^2(\tn^R): \wt D^R(f,f)<\infty\bigr\}.
\end{aligned}
\end{equation}
This form corresponds to the localized version of our process.

Next we move to truncation. Here we first show:

\begin{lemma}
For all $\kappa\in (0,1]$ and all $R\ge1$,
\begin{equation}\label{e:bound0}
\sup_{x\in \Z^d}\frac1{\tn^R(x)}\sum_{\begin{subarray}{c}y\in\Z^d\\|x-y|\le\kappa R\end{subarray}}\tC_{x,y}^R|x-y|^2
\le 1+2d
\end{equation}
and
\begin{equation}\label{e:bound}
\sup_{x\in  B(0,4R)}\frac{1}{\tn^R(x)} \sum_{\begin{subarray}{c}
y\in \Z^d\\|y-x|>\kappa R
\end{subarray}}
\tC_{x,y}^R
\le \frac{1+2d}{\kappa^2 R^2}.
\end{equation}
\end{lemma}

\begin{proofsect}{Proof}
By \eqref{E:2.1}, \eqref{E:3.11eq} and \eqref{E:3.12eq},
for all $R\ge1$  and all $x\in B(0, 2R)$,
\begin{equation}
\frac{1}{\tn^R(x)}\sum_{y\in\Z^d}\tC_{x,y}^R|x-y|^2= \frac{1}{\hn(x)}
\sum_{y\in\Z^d}C_{x,y}|x-y|^2= 1,
\end{equation}
while for $x\in B(0, 4R)\smallsetminus B(0,2R)$ we get
\begin{equation}\label{e:3.16}
\frac{1}{\tn^R(x)}\sum_{y\in\Z^d}\tC_{x,y}^R|x-y|^2 = \frac{1}{1+\hn(x)}
\Bigl(\sum_{y\in B(0,2R)}C_{x,y}|x-y|^2 + \sum_{\begin{subarray}{c}y\in B(0,2R)^\cc\\|x-y|=1\end{subarray}}1\Bigr)\le 1+2d.
\end{equation}
Hence,
\begin{equation}\label{e:3.17}
\sup_{R\ge1}\sup_{x\in B(0,4R)}\frac{1}{\tn^R(x)}\sum_{y\in\Z^d}\tC_{x,y}^R|x-y|^2\le1+2d.
\end{equation}
In particular, for all $\kappa\in (0,1]$ and all~$R\ge1$,
\begin{equation}
\label{e:bound0-}
\sup_{x\in  B(0,4R)}\frac{1}{\tn^R(x)}\sum_{\begin{subarray}{c}y\in\Z^d\\|x-y|\le\kappa R\end{subarray}}\tC_{x,y}^R|x-y|^2\le \sup_{x\in B(0,4R)}\frac{1}{\tn^R(x)}\sum_{y\in\Z^d}\tC_{x,y}^R|x-y|^2\le 1+ 2d.
\end{equation} On the other hand, \eqref{E:3.11eq} and \eqref{E:3.12eq} also give us that
for all $R\ge1$ and all $\kappa\in (0,1]$,
\begin{equation}
\label{e:bound0--}
\sup_{x\in B(0,4R)^c}\frac{1}{\tn^R(x)}\sum_{\begin{subarray}{c}y\in\Z^d\\|x-y|\le\kappa R\end{subarray}}
C_{x,y}^R|x-y|^2\le \sup_{x\in B(0,4R)^c}\sum_{\begin{subarray}{c}y\in\Z^d\\|x-y|=1\end{subarray}}1=2d.
\end{equation}
Combining \twoeqref{e:bound0-}{e:bound0--}, we have \eqref{e:bound0}. Noting that, in light of
{\eqref{e:3.17}}, the sum in \eqref{e:bound} is bounded~by
\begin{equation}
\frac{1}{\kappa^2 R^2}\sup_{x\in B(0,4R)}\frac{1}{\tn^R(x)} \sum_{y\in \Z^d}\tC_{x,y}^R |x-y|^2
\le \frac{1+2d}{\kappa^2 R^2},
\end{equation}
the claim follows.
\end{proofsect}

For all $\kappa\in (0,1]$ and all~$R\ge1$ satisfying $\kappa R\ge1$ we now define a
truncated, localized Dirichlet form $(\wt D^{R,\kappa}, \wt\Fs^R)$ by
\begin{equation}
\wt D^{R,\kappa}(f,f):=\sum_{\begin{subarray}{c}
x,y\in \Z^d\\ |x-y|\le\kappa R\end{subarray}}
\tC_{x,y}^R\bigl[f(y)-f(x)\bigr]^2,\quad
f\in \wt\Fs^R,
\end{equation}
which is well defined by \eqref{e:bound0}.
A starting point of our derivations is the following Sobolev inequality for $(\wt D^{R,\kappa}, \wt\Fs^R)$:

\begin{proposition}\label{P:Sob}
Let~$d\ge2$ and suppose  Assumptions~$\ref{ass1}$ and $\ref{Asmp2}$ hold.  There are $\epsilon\in (0,\frac4{d-2})$, a constant $c_1\in(0,\infty)$ and an a.s.-finite random variable $R_0:=R_0(\w)\ge1$ such that
\begin{equation}\label{e:Sob}
\|f\|^2_{\ell^{2+\epsilon}(\tn^R)}\le c_1 \left(R^{2-\frac{d\epsilon}{2+\epsilon}}\wt D^{R,\kappa}(f,f)+R^{-\frac{d\epsilon}{2+\epsilon}}\|f\|^2_{\ell^2(\tn^R)}\right)
\end{equation}
holds for all $\kappa\in (0,1]$, all
 $f\in \ell^2(\hn^R)$  and all  $R\ge R_0$ with~$\kappa R\ge1$.
\end{proposition}

The proof is based on two lemmas. Consider the Dirichlet form
\begin{equation}
D_1(f,f):=\sum_{\begin{subarray}{c}x,y\in\Z^d\\ |x-y|=1\end{subarray}}{\bar C_{x,y}}
\bigl[f(x)-f(y)\bigr]^2
\end{equation}
associated with an auxiliary collection $\{\bar C_{x,y}=\bar C_{y,x}\colon|x-y|=1\}$ of nearest-neighbor conductances.
We then have:

\newcommand{\bnu}{\bar\nu}

\begin{lemma}
\label{lemma0-1Mar}
For all~$d\ge2$, there is~$c(d)\in(0,\infty)$ and, for all~$p,q\in(\frac d2,\infty)$ satisfying \eqref{E:1.7*} there is $\epsilon\in(0,\frac4{d-2})$ with
\begin{equation}
\label{E:3.20eq}
\frac1q+\frac2{2+\epsilon}\frac1p=\frac2d-\frac{\epsilon}{2+\epsilon}
\end{equation}
such that for all~$L\ge1$, all
$f\colon\Z^d
\to[0,\infty)$ with $\supp(f)\subseteq B(0,L)$, all $\bnu\colon\Z^d\to[0,\infty)$ and all positive $\{\bar C_{x,y}=\bar C_{y,x}\colon|x-y|=1\}$,
\begin{equation}\label{E:3.23eq-}
\Bigl(\,\sum_{x\in \Z^d}f(x)^{2+\epsilon}\,\bnu(x)\Bigr)^{\frac2{2+\epsilon}}
\le c(d)
\Bigl(\frac{(2+\epsilon)p}{p-1}\Bigr)^2\,\alpha_L^{\frac2{p(2+\epsilon)}}\, \beta_L^{\frac{1}{q}} \,\, L^{\,2-\frac{d\epsilon}{2+\epsilon}}\,
D_1(f,f)
\end{equation}
holds with
\begin{equation}
\alpha_L:=\frac1{L^d}\sum_{x\in B(0,L)}\bnu(x)^p\quad\text{\rm and}\quad
\beta_L:=\frac1{L^d}\sum_{\begin{subarray}{c}
x\in B(0,L)\\y\colon |y-x|=1
\end{subarray}}(\bar C_{x,y})^{-q}.
\end{equation}
\end{lemma}

\begin{proofsect}{Proof}
Let~$p,q\in(\frac d2,\infty)$ obey \eqref{E:1.7*}. Then \eqref{E:3.20eq} is solved for~$\epsilon$ by
\begin{equation}\label{r:remark}
\epsilon := 2\,\Bigl(\frac2d-\frac1p-\frac1q\Bigr)\Bigl(\frac{d-2}d+\frac1q\Bigr)^{-1}\in\bigl(0,\tfrac4{d-2}\bigr).
\end{equation}
Let~$s>1$ be the index H\"older conjugate to~$p$.
Then by our restriction on the support of~$f$,
\begin{equation}
\label{E:1.1eq}
\sum_{x\in \Z^d}f(x)^{2+\epsilon}\,\bnu(x)
\le\Bigl(\,\sum_{x\in\Z^d}f(x)^{s(2+\epsilon)}\Bigr)^{1/s}
\alpha_L^{1/p}\, L^{d/p}.
\end{equation}
Next define~$r$ by
\begin{equation}
\label{E:1.3eq}
r\frac d{d-1}=s(2+\epsilon)
\end{equation}
and note that, since~$s>1$ and~$\epsilon>0$, we have~$r>1$.
Using that $|a^\gamma-b^\gamma|\le \gamma|a-b|(a^{\gamma-1}+b^{\gamma-1})$ holds for all~$a,b>0$ and all~$\gamma\ge1$,  the $\ell^1$-Sobolev inequality implies
\begin{equation}
\begin{aligned}
\label{E:1.4eq}
\Bigl(\,\sum_{x\in\Z^d}f(x)^{r\frac d{d-1}}\Bigr)^{\frac{d-1}d}
&\le c\sum_{|x-y|=1}\bigl| f(x)^r-f(y)^r\bigr|
\le 2c\,r\sum_{|x-y|=1}f(x)^{r-1}\bigl|f(x)-f(y)\bigr|,
\end{aligned}
\end{equation}
where $c$ is a $d$-dependent constant
(which is directly related to the isoperimetric constant on~$\Z^d$).

Define~$\theta\in(0,1/2)$ by $\theta:=\frac12(1-\frac1q)$ and note that, by \eqref{E:3.20eq},
\begin{equation}
\label{E:1.6eq}
r-1 = \theta r\frac d{d-1}.
\end{equation}
 The H\"older's inequality and the restriction on the support of~$f$ then give
\begin{equation}
\begin{aligned}
\label{E:3.26eq}
\sum_{|x-y|=1}f(x)^{r-1}\bigl|f(x)-f(y)\bigr|
&=\sum_{|x-y|=1}f(x)^{r\frac d{d-1}\theta}\,\bigl({\bar C_{x,y}}^{-\frac1{1-2\theta}}\bigr)^{\frac12-\theta}\,\Bigl(\bar C_{x,y} \bigl|f(x)-f(y)\bigr|^2\Bigr)^{\frac12}
\\&\le\Bigl(2d\sum_{x\in\Z^d}f(x)^{r\frac d{d-1}}\Bigr)^\theta
 \biggl(\sum_{\begin{subarray}{c}
x\in B(0,L)\\y\colon |y-x|=1
\end{subarray}}{\bar C_{x,y}}^{-\frac1{1-2\theta}}
\biggr)^{\frac12-\theta}
D_1(f,f)^{\frac 12}.
\end{aligned}
\end{equation}
Since \eqref{E:1.3eq} and \eqref{E:1.6eq} show
\begin{equation}
\frac{d-1}d-\theta = \frac1{s(2+\epsilon)}
\end{equation}
using $\theta\le1/2$ we can combine \eqref{E:3.26eq} with \eqref{E:1.4eq} to get
\begin{equation}
\label{E:3.29eq}
\Bigl(\,\sum_{x\in\Z^d}f(x)^{s(2+\epsilon)}\Bigr)^{\frac1{s(2+\epsilon)}}
\le
2c
(2d)^{\frac12}\,r\,L^{\,\frac d{2q}}\,\beta_L^{\frac1{2q}}\,
D_1(f,f)^{\frac12}.
\end{equation}
Plugging this in \eqref{E:1.1eq}, we obtain
\begin{equation}
\Bigl(\sum_{x\in \Z^d}f(x)^{2+\epsilon}\,\bnu(x)\Bigr)^{\frac2{2+\epsilon}}
\le 8d  c^2\,r^2 \,
L^{\frac{d}q+\frac{2d}{p(2+\epsilon)}}\,\alpha_L^{\frac2{p(2+\epsilon)}}\beta_L^{1/q}\,D_1(f,f).
\end{equation}
The claim now follows from \eqref{E:3.20eq} and (for the $r^2$ term) \eqref{E:1.3eq}.
\end{proofsect}

\begin{remark}\rm As noted by a referee, by \eqref{r:remark},
$$2+\epsilon=2\left(1-\frac{1}{p}\right)\frac{d}{d-2+d/q}.$$ The inequality \eqref{E:3.23eq-} in Lemma \ref{lemma0-1Mar} (with an imprecise constant) can be directly deduced from  a weighted version of Sobolev inequality given in \cite[(26) in Remark 3.6]{ADS}.
\end{remark}

For the next lemma, let
\begin{equation}\label{D_0def}
D_0(f,f):=\sum_{\begin{subarray}{c}x,y\in\Z^d\\ |x-y|=1\end{subarray}}
\bigl[f(x)-f(y)\bigr]^2
\end{equation}
denote the Dirichlet form associated with the simple random walk.
Recall that~$\mu$ is the counting measure on $\Z^d$.
Then we have:

\begin{lemma}\label{lemma0-2Kum}
For each~$d\ge2$ there is a constant~$c(d)\in(0,\infty)$ such that
for all $\epsilon\in(0,\frac4{d-2})$ and all
 $f\in \ell^2(\hm)$,
\begin{equation}
\label{E:functional-2}
\Bigl(\,\sum_{x\in \Z^d} f(x)^{2+\epsilon}\Bigr)^{\frac 2{2+\epsilon}}\le \bigl(c(d)\,
 (2+\epsilon)\bigr)^{\frac{\epsilon d}{2+\epsilon}}\,D_0(f,f)^{\frac{\epsilon d}{2(2+\epsilon)}} \,\Bigl(\,\sum_{x\in \Z^d} f(x)^2\Bigr)^{1-\frac{\epsilon d}{2(2+\epsilon)}}.
\end{equation}
\end{lemma}

\begin{proofsect}{Proof}
Let~$\epsilon\in(0,\frac4{d-2})$ and set
\begin{equation}
\label{E:3.32eq}
r:=\max\Bigl\{2,(2+\epsilon)\frac{d-1}d\Bigr\}.
\end{equation}
Then $r\frac d{d-1}>2$ and so there exist unique $\beta,\gamma\in\R$ such that
\begin{align}
 2+\epsilon&=\beta r\frac d{d-1}+(1-\beta)2,\label{MK211-1}\\
 r-1&=\gamma r\frac d{d-1}+\bigl(\tfrac 12-\gamma\bigr)2.
\label{MK211-2}
\end{align}
A calculation shows
\begin{equation}
\label{E:3.35eq}
\beta=\epsilon\Bigl(\frac{rd}{d-1}-2\Bigr)^{-1}\quad\text{and}\quad
\gamma=(r-2)\Bigl(\frac{rd}{d-1}-2\Bigr)^{-1}.
\end{equation}
 We claim that   $\beta\in(0,1]$ and~$\gamma\in[0,1/2)$. Indeed, $\beta>0$ and $\gamma\ge 0$ are immediate from \eqref{E:3.35eq} and $r\ge2$. The inequality $\beta\le 1$ is equivalent $r\ge \frac{d-1}d(2+\epsilon)$, which holds for our choice of~$r$ in \eqref{E:3.32eq}, while $\gamma < 1/2$ is equivalent to $r< 2\frac{d-1}{d-2}$. This requires $2+\epsilon< 2\frac d{d-2}$, which holds thanks to $\epsilon< \frac4{d-2}$.

 We first assume that $f\colon\Z^d
\to[0,\infty)$ {has} compact support.  Using \eqref{MK211-1} and H\"older's inequality we have
\begin{equation}\label{MK212-1}
\sum_{x\in \Z^d} f(x)^{2+\epsilon}\le
\Bigl(\,\sum_{x\in \Z^d} f(x)^{\frac{rd}{d-1}}\Bigr)^\beta
\Bigl(\,\sum_{x\in \Z^d} f(x)^{2}\Bigr)^{1-\beta}.
\end{equation}
Furthermore,  by the $\ell^1$-Sobolev inequality on~$\Z^d$, as in  \eqref{E:1.4eq} we obtain
\begin{equation}
\begin{aligned}
\Bigl(\sum_{x\in \Z^d} f(x)^{\frac{rd}{d-1}}\Bigr)^{\frac{d-1}d}
&\le  c\,r\sum_{x\in \Z^d} f(x)^{r-1}\bigl|f(x)-f(y)\bigr|\\
&\le  c\,rD_0(f,f)^{\frac 12}
\Bigl(\sum_{x\in \Z^d} f(x)^{\frac{rd}{d-1}}\Bigr)^\gamma
\Bigl(\sum_{x\in \Z^d} f(x)^{2}\Bigr)^{\frac 12-\gamma},
\end{aligned}
\end{equation}
where~$c\in(0,\infty)$ depends only on the spatial dimension~$d$ and where we relied on  \eqref{MK211-2} to get the second inequality. Hence we get
\begin{equation}\label{MK212-2}
\Bigl(\sum_{x\in \Z^d} f^{\frac{rd}{d-1}}(x)\Bigr)^{\frac{d-1}d-\gamma}
\le c\,rD_0(f,f)^{\frac 12}\Bigl(\sum_{x\in \Z^d} f^{2}(x)\Bigr)^{\frac 12-\gamma}.
\end{equation}
Noting that
\begin{equation}
\frac{d-1}d-\gamma = \frac\beta\epsilon\biggl[\frac{d-1}d\Bigl(\frac{rd}{d-1}-2\Bigr)-(r-2)\biggr]
=\frac{2\beta}{d\epsilon}
\end{equation}
and, after a short calculation, also
\begin{equation}
1-\beta+\bigl(\tfrac12-\gamma\bigr)\frac{d\epsilon}2 = \frac{2+\epsilon}2-\frac{\epsilon d}4,
\end{equation}
from \eqref{MK212-1} and \eqref{MK212-2} we conclude
\begin{equation}
\sum_{x\in \Z^d} f(x)^{2+\epsilon}\le  (c\,r)^{\frac{\epsilon d}2} D_0(f,f)^{\frac{\epsilon d}4} \Bigl(\,\sum_{x\in \Z^d} f(x)^2\Bigr)^{\frac{2+\epsilon}2-\frac{\epsilon d}4}.
\end{equation}
Raising both sides to $\frac2{2+\epsilon}$ and using the definition of~$r$,
the conclusion for all $f\colon\Z^d
\to[0,\infty)$ with compact support  follows.

 Next we suppose that $f\in \ell^2(\hm)$. Choose a sequence $\{f_n\}_{n\ge1}$  of functions with compact support such that $f_n$ converges to $f$ in $\ell^2(\hm)$ as $n\to \infty$. As $D_0(f,f)\le 8d \|f\|_{\ell^2(\hm)}^2$ for all $f\in \ell^2(\hm)$, we have
$D_0(f_n,f_n)\to D_0(f,f)$ as $n\to \infty$. Therefore, the conclusion for $f\in\ell^2(\hm)$ follows by applying $f_n$ into \eqref{E:functional-2} first and then letting $n\to \infty$.
\end{proofsect}

We remark that an alternative proof of Lemma~\ref{lemma0-2Kum} can be devised based on estimates for the transition probabilities of the simple random walk. In particular, \eqref{E:functional-2} is true even when~$d=1$
with $\varepsilon\in (0,\infty)$. On the other hand, the proof of Lemmas~\ref{lemma0-1Mar} and~\ref{lemma0-2Kum} becomes considerably easier in $d\ge3$ where one can rely on the $\ell^2$-Sobolev inequality.

\smallskip
With the above lemmas in hand, we are ready to give:

\begin{proof}[Proof of Proposition~$\ref{P:Sob}$]
Let~$d\ge2$ and suppose Assumption~\ref{Asmp2} holds. Since the inequality in \eqref{E:1.7*} is strict, we may assume that both indices are finite, i.e., $p,q\in(\frac d2,\infty)$. Under Assumption~\ref{ass1}, the Spatial Ergodic Theorem yields the existence of a constant~$c_0\in(0,\infty)$ and a random variable $R_0=R_0(\omega)$ (which may depend on $p,q$) with $\BbbP(1\le R_0<\infty)=1$ such that
\begin{equation}\label{e:note01b}
\forall R\ge R_0\colon\quad\sum_{x\in B(0, 16 R)} \hn(x)^p\le c_0 R^d
\quad\text{and}\quad
\sum_{\begin{subarray}{c}
x\in B(0, 16  R)\\y\colon|y-x|=1
\end{subarray}}
 (C_{x,y})^{-q}\le c_0 R^d.   %%We add ".".
\end{equation}
The definitions \twoeqref{E:3.11eq}{E:3.12eq} of $C_{x,y}^R$ and $\tn^R$ then give
\begin{equation}
\label{e:note01a}
\forall R\ge R_0\colon\quad
\sup_{x\in \Z^d}\sum_{z\in B(x, 8 R)} \tn^R(z)^p\le c_1 R^d
\quad\text{and}\quad
\sum_{\begin{subarray}{c}
x\in B(0, 8 R)\\y\colon|y-x|=1
\end{subarray}}
 (\tC^R_{x,y})^{-q}\le c_1 R^d
\end{equation}
for some $c_1\in (0,\infty)$
that also may depend on $p$ and $q$.

Let $\epsilon\in(0,\frac4{d-2})$ solve \eqref{E:3.20eq} and fix~$\kappa\in(0,1]$.
Lemma~\ref{lemma0-1Mar} with $\bar\nu(x):=\nu^R(x)$, $\bar C_{x,y}:=C_{x,y}^R$, $L:=
8 R$) along with \eqref{e:note01a}  shows
 the existence of a constant $c_2\in(0,\infty)$ that depends only on~$d$, $p$, $\epsilon$ and~$c_1$ above such that
\begin{equation}
\label{e:euq-1}
\left(\sum_{x\in \Z^d} f(x)^{2+\epsilon}\tn^R(x)\right)^{\frac2{2+\epsilon}}
\le c_2\,R^{2-\frac{d\epsilon}{2+\epsilon}}\,\wt D^{R,\kappa}(f,f)
\end{equation}
holds for all~$R\ge R_0$ with~$\kappa R\ge1$ and all~$f\colon\Z^d\to[0,\infty)$ such that~$\supp(f)\subseteq B(0, 8 R)$. Here we used that $D_1(f,f)\le\wt D^{R,\kappa}(f,f)$
due to~$\kappa R\ge1$ and the choice $\bar C_{x,y}:=C_{x,y}^R$.

Next we invoke Lemma~\ref{lemma0-2Kum} along with the fact that, for some constant~$c>0$,
\begin{equation}
r^{2-\frac{d\epsilon}{2+\epsilon}}a+r^{-\frac{d\epsilon}{2+\epsilon}}b
\ge c a \Bigl(\frac ba\Bigr)^{1-\frac{\epsilon d}{2(2+\epsilon)}}
\end{equation}
is valid for all~$a,b,r>0$, to get the existence of~$c_3\in(0,\infty)$ such that, for all~$f\colon\Z^d\to[0,\infty)$ with~$\supp(f)\subseteq B(0, 4 R)^\cc$,
\begin{equation}
\label{e:equ-2}
\begin{aligned}
\left(\,\sum_{x\in \Z^d} f(x)^{2+\epsilon}\tn^R(x)\right)^{\frac2{2+\epsilon}}
&\le c_3 R^{-\frac{d\epsilon}{2+\epsilon}}\left(R^2 D_0(f,f)+\sum_{x\in \Z^d} f(x)^2\right)
\\
&\le c_3 \left(R^{2-\frac{d\epsilon}{2+\epsilon}}  \wt D^{R,\kappa}(f,f)+R^{-\frac{d\epsilon}{2+\epsilon}}\sum_{x\in \Z^d} f(x)^2\tn^R(x)\right).
\end{aligned}
\end{equation}
Here we used  $\nu^R(x)=1$ for all $x\in B(0,4R)^c$ and  the definitions \twoeqref{E:3.11eq}{E:3.12eq} along with $\kappa R\ge 1$ to ensure that the conductance $\tC^R_{x,y}$ is no smaller than that of the simple random walk whenever~$x$ or~$y$ is in~$\supp(f)$.

Consider a mollifier $\phi_R\colon
 \Z^d \to[0,1]$ subject to
\begin{equation}
\phi_R(x)
\begin{cases}=1,&\quad x\in B(0,4 R),\\
\in [0,1],&\quad x\in B(0,8 R)\smallsetminus B(0,4  R),\\
=0,&\quad x\in B(0,8R)^\cc,
\end{cases}
\end{equation}
and
\begin{equation}
\bigl|\phi_R(x)-\phi_R(y)\bigr|\le \frac{|x-y|}{2R},\quad x,y\in \Z^d.
\end{equation}
Let $f\colon\Z^d\to[0,\infty)$. Since $\supp(f\phi_R)\subseteq B(0,8R)$ while $\supp(f(1-\phi_R))\subseteq B(0,4R)^\cc$, the bounds \eqref{e:euq-1} and \eqref{e:equ-2} show
\begin{equation}
\label{E:3.49eq}
\begin{aligned}
\biggl(&\,\sum_{x\in \Z^d} f(x)^{2+\epsilon} \tn^R(x)\biggr)^{\frac2{2+\epsilon}}
\\
&\le  2\left(\,\sum_{x\in \Z^d} \bigl(f(x)\phi_R(x)\bigr)^{2+\epsilon}\tn^R(x)\right)^{\frac2{2+\epsilon}}  +2\left(\sum_{x\in \Z^d} \bigl(f(x)(1-\phi_R(x))\bigr)^{2+\epsilon}\tn^R(x)\right)^{\frac2{2+\epsilon}}
\\
&\le  c_4 R^{2-\frac{d\epsilon}{2+\epsilon}}\Bigl[\wt D^{R,\kappa}(f\phi_R, f\phi_R)+\wt D^{R,\kappa} \bigl(f(1-\phi_R), f(1-\phi_R)\bigr)\Bigr]
\\
&\qquad\qquad\quad+c_4 R^{-\frac{d\epsilon}{2+\epsilon}}\sum_{x\in \Z^d} f(x)^2 \tn^R(x),
\end{aligned}
\end{equation}
where~$c_4:=2\max\{c_2,c_3\}$. For the sum of the two Dirichlet forms we then get
\begin{equation}
\begin{aligned}
\wt D^{R,\kappa}(f\phi_R, f\phi_R)+&\wt D^{R,\kappa} \bigl(f(1-\phi_R), f(1-\phi_R)\bigr)
\\
&\le 4\wt D^{R,\kappa}(f,f)+4\sum_{\begin{subarray}{c}
x,y\in\Z^d\\|x-y|\le \kappa R
\end{subarray}}
\tC^R_{x,y}\bigl[\phi_R(x)-\phi_R(y)\bigr]^2f(x)^2
\\
&\le  4\wt D^{R,\kappa}(f,f)+ R^{-2}\sum_{\begin{subarray}{c}
x,y\in\Z^d\\|x-y|\le \kappa R
\end{subarray}}
\tC^R_{x,y}|x-y|^2f(x)^2
\\
&\le 4\wt D^{R,\kappa}(f,f)+(1+2d)R^{-2}\sum_{x\in\Z^d}f(x)^2\tn^R(x),
\end{aligned}
\end{equation}
where \eqref{e:bound0} was used in the last inequality. Plugging this in \eqref{E:3.49eq}, the claim follows.
\end{proof}

We note that the above proof highlights the need for~$\hn$ as a reference measure
and its modification~$\nu^R$.

\subsection{Heat-kernel estimates}
 We will now apply the above functional inequalities to estimates of the heat kernels. Denote by~$Y^R:=\{Y^R_t:t\ge0\}$ the Hunt process associated with $(\wt
D^R,\wt\Fs^R)$ and let $p^R(t,x,y)$ be the associated transition probabilities. Similarly, write $Y^{R,\kappa}:=\{Y^{R,\kappa}_t: t\ge0\}$ for the Hunt process associated with $(\wt D^{ R,\kappa},\wt
\Fs^R)$ and let $p^{R,\kappa}(t,x,y)$ be the corresponding the transition probabilities.

\smallskip
We start a simple consequence of Proposition~\ref{P:Sob}:

\begin{lemma}
\label{C:Sob}
Suppose that Assumptions~$\ref{ass1}$ and $\ref{Asmp2}$ hold, and let~$\epsilon\in(0,\frac4{d-2})$ and the random variable $R_0:=R_0(\omega)$ be as in Proposition~$\ref{P:Sob}$. Then there exists a constant $c>0$
such that
\begin{equation}
\label{E:3.51eq}
p^{R,\kappa}(t, x,y)\le cR^{-d}\left(\frac{t}{R^2}\right)^{-\frac{2+\epsilon}{\epsilon}}\texte^{\frac12 tR^{-2}}\tn^R(y)
\end{equation}
holds for all $\kappa\in (0,1)$, all $R\ge R_0$ with~$\kappa R\ge1$, all $t>0$ and all $x,y\in \Z^d$.
\end{lemma}

\begin{proof}
Let $\epsilon\in(0,\frac4{d-2})$ and $R_0$ be as in Proposition~\ref{P:Sob}.
For $f\colon\Z^d\to[0,\infty)$, H\"{o}lder's inequality shows
\begin{equation}
\begin{aligned}
\sum_{x\in \Z^d} f(x)^2\tn^{R}(x)
&=\sum_{x\in \Z^d} f(x)^{\frac{2+\epsilon}{1+\epsilon}} f(x)^{\frac{\epsilon}{1+\epsilon}} \tn^{R}(x)
\\
&\le\left(\sum_{x\in \Z^d} f(x)^{2+\epsilon}\tn^R(x)\right)^{\frac1{1+\epsilon}}\left(\sum_{x\in \Z^d} f(x)\tn^R(x)\right)^{\frac\epsilon{1+\epsilon}}.  %%We add ".". 
\end{aligned}
\end{equation}
Pick~$\kappa\in(0,1)$ and assume $R\ge R_0$ with~$\kappa R\ge1$. Then \eqref{e:Sob} turns this into the Nash inequality
\begin{equation}
\label{e:Nash}
\Vert f\Vert_{\ell^2(\tn^R)}^{4\frac{1+\epsilon}{2+\epsilon}}
\le c_1R^{2-\frac{d\epsilon}{2+\epsilon}}\Bigl(\wt D^{R,\kappa}(f,f)+R^{-2}\|f\|^2_{\ell^2(\tn^R)}\Bigr)\,\Vert f\Vert_{\ell^1(\tn^R)}^{\frac{2\epsilon}{2+\epsilon}}.
\end{equation}
The general equivalence between heat-kernel bounds and the Nash inequality, cf Carlen, Kusuoka and Stroock~\cite[Theorem~(2.1)]{CKS}, states that the Nash inequality (for any real $n>0$)
\begin{equation}
\Vert f\Vert_2^{2+4/n}\le A\Bigl(D(f,f)+\delta\Vert f\Vert_2^2\Bigr)\Vert f\Vert_1^{4/n}
\end{equation}
leads to a uniform bound on the heat kernel by $(n A/t)^{n/2}\texte^{\frac12\delta t}$ --- which reflects the ``missing'' $1/2$ in our normalization of the Dirichlet form.  Applying
this to \eqref{e:Nash} with the specific parameter values $n:=2\frac{2+\epsilon}\epsilon$, $\delta:=R^{-2}$ and $A:=c_1R^{2-\frac{d\epsilon}{2+\epsilon}}$, we get \eqref{E:3.51eq}.
\end{proof}

The inequality \eqref{E:3.51eq} is particularly useful when~$t$ and~$R$ are related by diffusive scaling and it provides a version of a uniform, a.k.a.\ diagonal, heat-kernel upper bound. For the off-diagonal estimate, we have to work somewhat harder:

\begin{proposition}
\label{thm:hitest}
Suppose Assumptions~$\ref{ass1}$ and $\ref{Asmp2}$ hold, and let~$\epsilon\in(0,\frac4{d-2})$ and the random variable $R_0:=R_0(\omega)$ be as in Proposition~$\ref{P:Sob}$. For every $\kappa\in (0,1]$,
there is a constant $c\in(0,\infty)$ such that for all $x,y\in \Z^d$, all $R\ge R_0(\omega)$ with $\kappa R\ge1$ and all $0<t\le R^2$,
\begin{equation}
\label{qtesr}
p^{R,\kappa}(t,x,y)\le c R^{-d}\Bigl(\frac{t}{R^2}\Bigr)^{-\frac{2+\epsilon}{\epsilon}} \exp\left(-\frac{|x-y|}{5\kappa R}\log \left(\frac{R^2}{t}\right)\right)\tn^R(y).
\end{equation}
\end{proposition}

\begin{proof}
We will invoke an argument from Carlen, Kusuoka and Stroock~\cite{CKS} (based on an earlier argument of Davies~\cite{Davies}) for obtaining off-diagonal heat-kernel bounds from
the Nash inequality \eqref{e:Nash}.
For that we first introduce the auxiliary objects
\begin{equation}
\begin{aligned}
\Gamma _{R}(\psi )(x)& :=\frac1{\hn^R(x)}\sum_{y\in \Z^d} (\texte^{\psi (x)-\psi (y)}-1)^{2}\tC^R_{xy}\IA_{\{|x-y|\le \kappa R\}}, \\
\Lambda (\psi )^{2}& :=\bigl\Vert \Gamma _{R}(\psi )\bigr\Vert _{\infty }\vee \bigl\Vert
\Gamma _{R}(-\psi )\bigr\Vert _{\infty }, \\
E_{R}(t,x,y)&:=\sup \bigl\{|\psi (x)-\psi (y)|-t\Lambda (\psi )^{2}:\ \Lambda (\psi )<\infty \bigr\}.
\end{aligned}
\end{equation}
Here~$\psi$ can be chosen to be any bounded function in the domain of the Dirichlet form $(\wt D^{ R,\kappa},\wt\Fs^R)$. In particular, we can take $\psi$ with bounded support.
Carlen, Kusuoka and Stroock~\cite[Theorem~(3.25)]{CKS} then shows that there is a constant $c_0>0$ such that for all  $\kappa\in (0,1]$, all $R\ge R_0(\omega)$ with~$\kappa R\ge1$, all $t>0$ and all $x,y\in \Z^d$,
\begin{equation}
\label{(7.2)}
p^{R,\kappa}(t,x,y)\leq c_0
\, R^{-d}\left(\frac{t}{R^2}\right)^{-\frac{2+\epsilon}{\epsilon}}
\texte^{ \frac12 tR^{-2}}
\texte^{-E_{R}(2t,x,y)}\tn^R(y)
\end{equation}
which, we note, refines the estimate from Lemma~\ref{C:Sob}.
In order to bring \eqref{(7.2)} into  the desired
form,
it thus suffices to supply a good lower bound on $E_{R}(2t,x,y)$.

 Fix $x_0,y_0\in \Z^d$, let $\lambda\ge0$ and consider the test function
\begin{equation}
\psi (x):=\lambda\bigl(|x_0-y_0|-|x_{0}-x|\bigr)_{+}.
\end{equation}
The triangle inequality gives $|\psi (x)-\psi (y)|\leq \lambda |x-y|$. According to the elementary inequalities $|\texte^{t}-1|^{2}\leq
t^{2}\texte^{2|t|}$ and $t^2\texte^{-|t|}\le2$ for $t\ge0$, we then  get from \eqref{e:bound0} that
\begin{equation}
\begin{aligned}
\Gamma _{R}(\psi)(x) &=\frac1{\hn^R(x)}
\sum_{y\in \Z^d} (\texte^{\psi (x)-\psi (y)}-1)^{2}\tC^R_{xy}\IA_{\{|x-y|\le\kappa R\}}\\
&\leq \frac1{\hn^R(x)}\texte^{4\lambda \kappa R}\lambda ^{2}\sum_{y\in \Z^d} |x-y|^{2}\tC^R_{xy}\IA_{\{|x-y|\le \kappa R\}}
\\
&\leq  (1+2d) \lambda^{2}\texte^{4\lambda \kappa R}\leq  2(1+2d) \kappa^{-2}\texte^{5\lambda
\kappa R}R^{-2}.
\end{aligned}\end{equation}
Since the same bound applies to $\Gamma_R(-\psi)$ as well, the fact that $\psi(x_0)=|x_0-y_0|$ while $\psi(y_0)=0$ shows
\begin{equation}
-E_{R}(2t,x_{0},y_{0})\leq -\lambda |x_0-y_0|+2(1+2d) t\kappa^{-2}\texte^{5\lambda \kappa R}R^{-2}.
\end{equation}
Suppose that $0<t\le R^2$ and set
\begin{equation}
\lambda:=\frac{1}{5\kappa R}\log \Bigl(\frac{R^{2}}{t}\Bigr).
\end{equation}
 Then
\begin{equation}
-E_{R}(2t,x_{0},y_{0}) \le
2(1+2d)\kappa^{-2}- \frac{|x_0-y_0|}{5\kappa R}\log\Bigl(\frac{R^{2 }}{t}\Bigr).
\end{equation}
 Denoting $c:=\texte^{\frac12+2(1+2d)\kappa^{-2}}c_0$, the claim now follows from \eqref{(7.2)}.
\end{proof}

\subsection{ Exit time estimates}
\label{3-3}\noindent
 The uniform estimate on the transition probabilities of the truncated, localized process~$Y^{R,\kappa}$ permits us to control the tails of the exit times thereof. This can then be extended to the process~$Y$ as well. Indeed, given~$A\subseteq\Z^d$, define the first exit time from~$A$ by
\begin{equation}
\tau_A:=\inf\{t>0: Y_t \notin A\}.
\end{equation}
We then have:

\begin{proposition}
\label{P:ex}
Under Assumptions~$\ref{ass1}$ and $\ref{Asmp2}$, there is
a random variable $R_1:=R_1(\omega)$ with $\BbbP(1\le R_1<\infty)=1$ and,  for each~$\delta\in(0,1]$,  a constant $c\in(0,\infty)$ such that for all $t>0$, all $R\ge \delta^{-1}R_1$ and all $x\in B(0,R)$,
\begin{equation}\label{e0}
P^x(\tau_{B(x,\delta R)}\le t)\le \frac{ct}{R^2}.
\end{equation}
\end{proposition}

The proof is based on a comparison with the corresponding exit problems for the walks~$Y^R$ and~$Y^{R,\kappa}$.  For all $R\ge1$, all $\kappa\in (0,1]$, all $x\in \Z^d$  and all $r\ge1$, let
\begin{equation}
\tau^{R,\kappa}_{B(x,r)}:=\inf\bigl\{t\ge0\colon Y_t^{R,\kappa}\notin B(x,r)\bigr\}.
\end{equation}
We then have:

\begin{lemma}
\label{C:exit}
Suppose Assumptions~$\ref{ass1}$ and $\ref{Asmp2}$ hold and let~$R_0:=R_0(\w)$ be as in Proposition~$\ref{P:Sob}$. There is
$\kappa\in (0,1]$ and, for each~$\delta\in(0,1]$, also~$c>0$ such that for all  $x\in \Z^d$, all $R\ge \max\{16 \delta^{-1}R_0,(\kappa\delta)^{-1}\}$ and all $t>0$,
\begin{equation}
P^x\bigl(\tau_{B(x,\delta R)}^{R,\delta\kappa}\le t\bigr)\le \frac{ct}{R^2}.
\end{equation}
\end{lemma}

\begin{proof}
Let~$\epsilon\in(0,\frac4{d-2})$ and~$R_0$ be as in Proposition~\ref{P:Sob} and let~$\kappa\in(0,1]$ be such that
\begin{equation}
\label{E:3.65eq}
\frac1{20\kappa}-\frac{2+\epsilon}\epsilon\ge1.
\end{equation}
Fix~$\delta\in(0,1]$. Since \eqref{qtesr} applies with~$\kappa$ replaced by~$\kappa\delta$ for all $R\ge R_0$ satisfying~$\kappa\delta R\ge1$, all $0<t\le R^2$ and all~$x\in\Z^d$, we get
\begin{equation}
\label{E:3.66eq}
\begin{aligned}
P^x\bigl(&|Y^{R,{ \delta }\kappa}_t-x|\ge  \tfrac12 \delta R\bigr)
\\
&\le c_0 \left(\frac{t}{R^2}\right)^{-\frac{2+\epsilon}\epsilon} R^{-d} \sum_{\begin{subarray}{c}y\in \Z^d\\|y-x|\ge  \frac12 \delta R\end{subarray}} \exp\left(-\frac{|x-y|}{5 \delta \kappa R}\log \left(\frac{R^2}{t}\right)\right)\tn^{R}(y)\\
&=c_0\left(\frac{t}{R^2}\right)^{-\frac{2+\epsilon}\epsilon} R^{-d}\sum_{n=1}^\infty\sum_{\begin{subarray}{c}y\in \Z^d\\
n\le 2\frac{|y-x|}{ \delta  R}< n+1\end{subarray}}
\exp\left(-\frac{|x-y|}{5 \delta \kappa R}\log \left(\frac{R^2}{t}\right)\right)\tn^{R}(y),
\end{aligned}
\end{equation}
where $c_0$ depends on $\kappa$ and $\delta$.
Assuming in addition that $t\le R^2/3$ (which ensures $\log(R^2/t)\ge1$) and
noting that the condition $ \delta R\ge 16R_0(\omega)$
enables us to apply \eqref{e:note01b} and
\eqref{e:note01a}, the two sums on the right are now bounded by
\begin{equation}
\label{E:3.67eq}
\begin{aligned}
\sum_{n=1}^\infty \exp\biggl(-\frac{n}{10\kappa}\log \Bigl(\frac{R^2}{t}\Bigr)\biggr)&\biggl(\,\sum_{y\in B(x,\frac12(n+1)\delta R)}\tn^R(y)\biggr)
\\
&\le c_1( \delta  R)^d\sum_{n=1}^\infty n^d\exp\left(-\frac{n}{10\kappa}\log \left(\frac{R^2}{t}\right)\right)
\qquad\qquad\quad
\\
&\le c_2 R^d \exp\left(-\frac{1}{ 20 \kappa}\log \left(\frac{R^2}{t}\right)\right)
=c_2R^d\left(\frac{t}{R^2}\right)^{\frac1{20\kappa}}
\end{aligned}
\end{equation}
 for some~$c_2$ depending on~$\kappa$ and $\delta$. Combining \twoeqref{E:3.66eq}{E:3.67eq}, from \eqref{E:3.65eq} we obtain
\begin{equation}
\label{E:3.70eq}
P^x\bigl(|Y^{R,{ \delta }\kappa}_t-x|\ge \tfrac12\delta R\bigr)\le  c_3\frac{t}{R^2}.
\end{equation}
for all~$x\in\Z^d$, all~$R\ge R_0$ with $\kappa\delta R\ge1$, $\delta R\ge 16R_0(\omega)$  and all~$t$  with~$0<t<R^2/3$.

The strong Markov property  at the first exit time from $B(x, {\delta} R)$  shows
\begin{equation}
\label{E:3.68eq}
\begin{aligned}
P^x\bigl(\tau_{B(x, {\delta} R)}^{R,{\delta}\kappa}\le t\bigr)
&\le P^x\bigl(|Y^{R,{\delta}\kappa}_{2t}-x|\ge  \tfrac12{\delta}R\bigr)
+P^x\Bigl(|Y^{R,{\delta}\kappa}_{2t}-x|\le \tfrac12{\delta} R,\tau_{B(x,{\delta}R)}^{R,{\delta}\kappa}\le t\Bigr)
\\
&\le P^x\bigl(|Y^{R,{\delta}\kappa}_{2t}-x|\ge \tfrac12{\delta}R\bigr)+\sup_{z\in \Z^d}\sup_{0<s\le 2 t}P^z\bigl(|Y^{R,{\delta}\kappa}_{2t-s}-z|\ge  \tfrac12 {\delta}R\bigr).
\end{aligned}
\end{equation}
 Invoking \eqref{E:3.70eq}, we get the claim for all~$t<R^2/6$. Adjusting the constant~$c$ if necessary, the claim holds trivially for~$t\ge R^2/6$.
\end{proof}

 Next we set
\begin{equation}
\tau^R_A:=\inf\{t\ge0\colon Y^R_t \notin A\}.
\end{equation}
Then we get the following deterministic estimate:

\begin{lemma}
\label{nl-1}
 There is $c>0$ such that all $t>0$, all $\kappa\in (0,1]$,  all~$R\ge1$ with $\kappa R\ge1$, all $r\ge1$ and all $x\in \Z^d$,
\begin{equation}\label{nl1-2}
\Bigl|P^x\bigl(\tau^R_{B(x,r)}\le
t\bigr)-P^x\bigl(\tau^{R,\kappa}_{B(x,r)}\le t\bigr)\Bigr|
\le ct\sup_{y\in B(x,r)}\frac{1}{\tn^R(y)}\sum_{\begin{subarray}{c}z\in \Z^d\\|y-z|>\kappa R\end{subarray}} \tC^R_{y,z}.
\end{equation}
\end{lemma}

\begin{proof}
 This is  proved by following the argument of  \cite[Lemma~3.1]{CKW},  which is itself based on the
Meyer's construction of $Y^R$ (see \cite[Section~3.1]{BGK}).
\end{proof}

We are ready to give:

\begin{proofsect}{Proof of Proposition~$\ref{P:ex}$}
Fix $\kappa\in (0,1]$ in $(\wt D^{R,\kappa},\wt\Fs^R)$ to the constant in Lemma~\ref{C:exit}.
Since the processes~$Y$ and~$Y^R$ ``see'' the same conductances  in~$B(0,2R)$,  for all~$R\ge r\ge1$, all~$t>0$ and all $x\in B(0,R)$ we have
\begin{equation}
P^x\bigl(\tau_{B(x,r)}\le
t\bigr)=P^x\bigl(\tau^R_{B(x,r)}\le
t\bigr).
\end{equation}
Lemma~\ref{nl-1} along with \eqref{e:bound} then show
\begin{equation}
\Bigl|P^x\bigl(
\tau_{B(x,r)}
\le t\bigr)-P^x\bigl(\tau^{R,\delta\kappa}_{B(x,r)}\le t\bigr)\Bigr|\le \frac{ct}{R^2}
\end{equation}
for all~$R\ge r\ge1$ with $\delta\kappa R\ge1$, all $t>0$ and all~$x\in\Z^d$, where~$c$ depends on the constant from \eqref{nl1-2}, $\kappa$ and $\delta$.  Setting $r:=\delta R$, Lemma~\ref{C:exit} gives the claim with  $R_1:=16
 R_0/\kappa$.
\end{proofsect}

 Let $p^{B(0,R)}(t,x,y)$ denote the  (substochastic) transition probabilities of the process $Y$ killed upon exiting the ball $B(0,R)$. The conclusions for the heat-kernel associated with the truncated, localized process $Y^{R,\kappa}$ can then be transferred to~$Y$:

\begin{proposition}\label{P:heat}
Under Assumptions~$\ref{ass1}$ and $\ref{Asmp2}$, and with~$\epsilon\in(0,\frac4{d-2})$  and $R_0$  as in Proposition~$\ref{P:Sob}$, there is a constant $c>0$
such that for all $R\ge R_0$, all $x,y\in B(0,R)$ and all $0<t\le R^2$,
\begin{equation}
p^{B(0,R)}(t,x,y)\le c\,R^{-d}\left(\frac{t}{R^{2}}\right)^{-\frac{2+\epsilon}\epsilon}\,\hn(y).
\end{equation}
\end{proposition}

\begin{proofsect}{Proof}
Denote by $p^{R, B(0,R)}$ the transition probabilities of the process $Y^R$ killed upon exiting the ball $B(0,R)$. Then for all $x,y\in B(0,R)$ and all $t>0$,
\begin{equation}
p^{B(0,R)}(t,x,y)=p^{R,B(0,R)}(t,x,y).
\end{equation}
 Since, trivially,
\begin{equation}
p^{R,B(0,R)}(t,x,y)\le p^{R}(t,x,y)
\end{equation}
 it suffices to prove the desired bound for the transition probabilities $p^R(t,x,y)$ of the process~$Y^R$. Here we note that the associated Dirichlet forms obey  $\wt D^{R,\kappa}(f,f)\le \wt D^{R}(f,f)$ and so the Nash inequality \eqref{e:Nash} applies for the Dirichlet form $(\wt D^{R}, \wt
\Fs^R)$ as well. Since $\tn^R=\nu$ on~$B(0,R)$, the argument from the proof of Lemma~\ref{C:Sob} then gives the claim.
\end{proofsect}

\section{Proof of Quenched Invariance Principle}
\noindent
Having established the needed bounds on the transition probabilities and exit times, we proceed to the proof the quenched invariance principle.

\subsection{Tightness}
\label{3-4}\noindent
We start with the proof of tightness of diffusively-scaled process~$Y$. Our aim is to apply the criterion for tightness from Aldous~\cite{Aldous}. Unfortunately, this result if not formulated for the space $\CC([0,T])$ but rather for the Skorohod space $\DD([0,T])$ of functions $f\colon [0,T]\to \R^d$ that are right continuous on $[0,T)$ and have left limits on $(0,T]$. This space can be endowed with the standard Skorohod topology (see, e.g., Billingsley~\cite{Billingsley}) that makes it a Polish space which in turn permits considerations of weak limits of probability measures.

A $d$-dimensional version of Aldous~\cite[Theorem~1]{Aldous} then implies that the sequence $\{Y^{(n)}\colon n\ge1\}$ of processes is tight in~$\DD([0,T])$ when the following two conditions are met:
\begin{enumerate}
\item[(1)] $\{Y_t^{(n)}\colon n\ge1\}$ is tight, as $\R^d$-valued random variables, for each~$t\in[0,T]$, and
\item[(2)] for any
sequence $\{\tau_n:n\ge1\}$, where $\tau_n$ is for each $n\ge1$ a stopping time for
the natural filtration of~$Y^{(n)}$, and any~$\delta_n>0$ with $\delta_n\to0$,
\begin{equation}
Y^{(n)}_{(\tau_n+\delta_n)\wedge T}-Y^{(n)}_{\tau_n\wedge T}\,\underset{n\to\infty}\longrightarrow\,0
\end{equation}
in probability.
\end{enumerate}
We will apply this to the choice
\begin{equation}\label{E-Yn}
Y_t^{(n)}:=\frac1{\sqrt n}\,Y_{nt},\qquad t\ge0,
\end{equation}
to get:

\begin{proposition}\label{thm:tightthm}
Let $d\ge 2$. For each $T>0$ and a.e.\ realization of the conductances, the laws of $\{Y^{(n)}
\colon n\ge1\}$ induced by~$P^0$ on $\DD([0,T])$ are tight.
\end{proposition}
\begin{proofsect}{Proof}
Let $R_1=R_1(\omega)$ be as in Proposition \ref{P:ex}.
We will check that the above conditions (1-2) from Aldous~\cite{Aldous} hold on the set $\{R_1<\infty\}$.
To distinguish various processes, let us write $\tau_B(X)$ for the first exit time of the process~$X$ from set~$B$.

For (1) we note that, by \eqref{e0} in Proposition \ref{P:ex}, when $r\sqrt n\ge  R_1$,
\begin{equation}
\label{E:3.62}
P^0\bigl(|Y^{(n)}_t|>r\bigr)
\le
P^0\bigl(\tau_{B(0,r)}(Y^{(n)})\le t\bigr)
=P^{0}\bigl(\tau_{B(0,r\sqrt n)}(Y)\le nt\bigr)
\le c_1 t/r^2
\end{equation}
This implies condition~(1) above on~$\{R_1<\infty\}$.

Next, pick $T>0$ and $\eta>0$, let $\tau_n$ be stopping times
bounded by $T$ and choose $\delta_n$ with $\delta_n\downarrow 0$. For any~$r>0$, the strong Markov property gives
\begin{multline}
\quad
P^0\bigl( |{Y}^{(n)}_{\tau_n+\delta_n}-{Y}^{(n)}_{\tau_n}|>\eta\bigr)
\\
\le P^0\bigl(\tau_{B(0,r)}(Y^{(n)})\le T\bigr)+\max_{z\in B(0,r\sqrt n)}P^z\bigl( |{Y}^{(n)}_{\delta_n}-z/\sqrt n|>\eta\bigr).
\quad
\end{multline}
Using \eqref{E:3.62}, the first quantity on the right is at most $c_1T/r^2$ whenever $r\sqrt n\ge
 R_1$.  For the second quantity Proposition \ref{P:ex} with $\delta:=\eta/r$ gives
\begin{equation}
\max_{z\in B(0,r\sqrt n)}P^z\bigl( |{Y}^{(n)}_{\delta_n}-z/\sqrt n|>\eta\bigr)
\le \max_{z\in B(0,r\sqrt n)}P^z\bigl(\tau_{B(z,\eta\sqrt n)}(Y)<n\delta_n\bigr)
\le c_2\delta_n/r^2
\end{equation}
for  $\min\{r\sqrt n,\eta\sqrt n\} \ge  R_1$.  While~$c_2$ depends on the ratio~$\eta/r$, for~$\eta$ and~$r$ fixed the right-hand side tends to zero in light of~$\delta_n\to0$. Thus we get
\begin{equation}
\limsup_{n\to\infty}\,P^0\bigl( |{Y}^{(n)}_{\tau_n+\delta_n}-{Y}^{(n)}_{\tau_n}|>\eta\bigr)\le c_3\,T/{r^2}\qquad \text{on }\{R_1<\infty\}
\end{equation}
for some constant~$c_3\in(0,\infty)$ regardless of $\eta$ or the choice of stopping times $\tau_n$ (as long as $\tau_n\le T$). But the left-hand side does not depend on~$r$ and so taking $r\to\infty$, we obtain condition~(2) above on~$\{R_1<\infty\}$.
Aldous~\cite[Theorem~1]{Aldous} then implies tightness of the processes $\{Y^{(n)}:n\ge1\}$ on~$\DD([0,T])$.
\end{proofsect}

\subsection{Proof of a QIP}\label{3-5}
Having proved tightness, our proof of a QIP is now reduced to the convergence of finite-dimensional distributions. Let $\{\wt B_t\colon t\ge0\}$ denote a $d$-dimensional Brownian motion such that
\begin{equation}
\E(\wt B_t)=0\quad\text{and}\quad
\E\bigl((v\cdot \wt B_t)^2\bigr)=t\,\,\frac{\E\pi(0)}{\E\hn(0)}\,v\cdot\Sigma v,\qquad v\in\R^d,\,t\ge0,
\end{equation}
where~$\Sigma$ is the matrix with entries as in~\eqref{E:sigma}. Then we have:

\begin{proposition}
\label{prop-FDD}
Let~$d\ge2$ and consider the processes $\{Y^{(n)}:n\ge1\}$ from \eqref{E-Yn} with law~$P^0$. Then the following holds on a set of conductances of full $\BbbP$-measure: For each $k\ge1$ and each $t_1,\dots,t_k$ satisfying $0\le t_1<t_2<\dots<t_k<\infty$,
\begin{equation}
\label{E:3.60a}
\bigl(Y^{(n)}_{t_1},\dots, Y^{(n)}_{t_k}\bigr)\,\underset{n\to\infty}{\overset{\text{\rm law}}\longrightarrow}\,
(\wt B_{t_1},\dots,\wt B_{t_k}).
\end{equation}
\end{proposition}

\begin{proofsect}{Proof}
One of the main issues in the proof is a proper demonstration of the set of conductances of full~$\BbbP$-measure on which \eqref{E:3.60a} holds for all~$k$-tuples $(t_1,\dots,t_k)$ with the stated properties. We will therefore keep careful track of all requisite events.

Let~$\Psi(x)$ denote the ``harmonic coordinate'' function from \eqref{E:2.5}; this is defined (and depends on) conductances in a measurable set~$\Omega_1$ with~$\BbbP(\Omega_1)=1$. Given a realization of the conductances and a path $Z:=\{Z_n:n\ge1\}$ of the discrete-time Markov chain, consider the random variables $\{\Psi(Z_n)\}$. A classical argument (cf, e.g., Corollary~3.10 of Biskup~\cite{Biskup-review}) based on the fact that $\Psi(Z_n)$ is a martingale implies that, under our standing assumptions, there is a measurable set~$\Omega_2\subseteq\Omega_1$ with~$\BbbP(\Omega_2)=1$ such that for each realization of conductances in~$\Omega_2$, the law of
\begin{equation}
t\mapsto \frac1{\sqrt n}\Psi(Z_{\lfloor tn\rfloor})
\end{equation}
induced by~$P^0$ on $\DD([0,T])$ --- in fact, even on $\CC([0,T])$, provided we interpolate values linearly --- tends to Brownian motion with mean zero and covariance~$\Sigma$.

Next we will prove a similar statement for $t\mapsto\frac1{\sqrt n}\Psi(Y_{nt})$ but for that we have to control the time change that takes~$Z$ into~$Y$. To that end, conditionally on~$Z$, let $T_0,T_1,\dots$ denote independent exponentials with parameters $\pi(Z_0)/\hn(Z_0),\pi(Z_1)/\hn(Z_1),\dots$, respectively. Then $\{\wt Y_t\colon t\ge0\}$, defined by
\begin{equation}
\label{E:3.66q}
\wt Y_t:=Z_{N_t}\quad\text{for}\quad N_t:=\max\{k\ge0\colon T_1+\dots+T_k\le t\},
\end{equation}
has the law of $\{Y_t\colon t\ge0\}$.
Letting $\Omega_3\subseteq\Omega_2$ be the subset of conductances on which
\begin{equation}
\lim_{n\to\infty}\,\frac1n\sum_{k=0}^n\frac{\hn(Z_k)}{\pi(Z_k)}=\E_{\Q_Z}\Bigl(\frac{\hn(0)}{\pi(0)}\Bigr),\qquad P^0\text{-a.s.},
\end{equation}
and
\begin{equation}
\label{E:3.63qq}
\forall\epsilon>0\colon\qquad
\lim_{n\to\infty}\,\frac1n\sum_{k=0}^n\frac{\hn(Z_k)}{\pi(Z_k)}\IA_{\{\hn(Z_k)/\pi(Z_k)>\epsilon n\}}=0,\qquad P^0\text{-a.s.}
\end{equation}
The stationarity and ergodicity of $\Q_Z$ with respect to the chain on environments induced by~$Z$ guarantees (via the Pointwise Ergodic Theorem) that $\BbbP(\Omega_3)=1$. Invoking the Weak Law of Large Numbers (with a simple truncation step enabled by \eqref{E:3.63qq}) and a renewal argument, we then have
\begin{equation}
\label{E:3.69q}
\frac{N_t}t\,\underset{t\to\infty}\longrightarrow\,\frac{\E\pi(0)}{\E\hn(0)}\qquad\text{in $P^0$-probability}
\end{equation}
for all conductances from~$\Omega_3$. In light of monotonicity of $t\mapsto N_t$, this gives a locally-uniform closeness of $s\mapsto N_{ts}/t$ to a linear function. By the definition of the Skorohod topology, the identification $\wt Y\laweq Y$ now shows that also the law
\begin{equation}
t\mapsto\frac1{\sqrt n}\Psi(Y_{nt})
\end{equation}
induced by $P^0$ on $\DD([0,T])$ tends to that of a Brownian motion with mean zero and covariance $(\E\pi(0)/\E\hn(0))\Sigma$, for every realization of conductances in~$\Omega_3$.

Since convergence on $\DD([0,T])$ to a process with continuous paths implies convergence of finite-dimensional distributions, to get \eqref{E:3.60a} it now suffices to identify a measurable set $\Omega^\star\subseteq\Omega_3$ of conductances with $\BbbP(\Omega^\star)=1$ such that
\begin{equation}
\label{E:3.63}
\frac1{\sqrt n}\,\,|\chi(Y_{tn})|
\,{\underset{n\to\infty}\longrightarrow}\,0\qquad\text{in $P^0$-probability,}
\end{equation}
holds on~$\Omega^\star$ for each $t\ge0$. For this
we argue as follows. For any $\eta>0$,
\begin{equation}
\label{E:4.16eq}\begin{aligned}
P^0\bigl(|\chi(Y_{tn})|\ge \eta \sqrt{n}\bigr)\le& P^0\bigl(\tau_{B(0,\eta^{-1}\sqrt{n})}(Y)\le tn\bigr)\\
&+\sum_{|x|\le \eta^{-1}\sqrt{n}}\IA_{\{|\chi(x)|>\eta \sqrt{n}\}} P^0\bigl(\tau_{B(0,\eta^{-1}\sqrt{n})}(Y)> tn, Y_{tn}=x\bigr).
\end{aligned}
\end{equation}
Assume that $\eta$ is so small that  $t\le \eta^{-2}$.
By Proposition \ref{P:ex},
\begin{equation}
P^0\bigl(\tau_{B(0,\eta^{-1}\sqrt{n})}(Y)\le tn\bigr)\le c_1\eta^2t \quad\text{whenever}\quad
\eta^{-1} \sqrt{n}\ge  R_1
\end{equation}
for some $c_1$ independent of $n$, $\eta$ and $t$,
 where $R_1=R_1(\omega)$ is as in Proposition \ref{P:ex}.
The contribution of this term to \eqref{E:4.16eq} thus vanishes as~$n\to\infty$ followed by~$\eta\downarrow0$. Let $R_0=R_0(\omega)$ be as in Proposition \ref{P:Sob}. Proposition~\ref{P:heat} in turn gives
\begin{equation}
P^0\bigl(\tau_{B(0,\eta^{-1}\sqrt{n})}(Y)> tn, Y_{tn}=x\bigr)
\le  c_2\eta^d(t\eta^2)^{-\frac{2+\epsilon}\epsilon} n^{-d/2}\tn(x)
\end{equation}
 whenever $\eta^{-1} \sqrt{n}\ge R_0$ and $t\le \eta^{-2}$, where  $c_2$ is independent of $n$, $x$, $\eta$ and $t$.
Using H\"{o}lder's inequality with~$p$ as Assumption~\ref{Asmp2}, the sum on the right of \eqref{E:4.16eq} is thus bounded by a constant times
\begin{equation}
(t\eta^2)^{-\frac{2+\epsilon}\epsilon} \left(\eta^dn^{-d/2}\sum_{|x|\le\eta^{-1}\sqrt{n}} \hn(x)^{p}\right)^{1/p}\left(\eta^dn^{-d/2}\sum_{|x|\le\eta^{-1}\sqrt{n}}\IA_{\{|\chi(x)|>\eta \sqrt{n}\}}\right)^{1-1/p}.
\end{equation}
The~$p$-integrability of~$\hn$ ensures that the term in the first large parentheses is bounded uniformly
in $n\ge1$,  $\BbbP$-a.s.  Thanks to corrector sublinearity on average~\eqref{E:sub}, the term in the second large parentheses, and thus the whole expression, tends to zero $\BbbP$-a.s.\ as $n\to \infty$. This proves \eqref{E:3.63} and thus the whole claim.
\end{proofsect}

Let us now see how the above proposition implies our main result:

\begin{proofsect}{Proof of Theorem~$\ref{thm:1}$}
Fix~$T>0$. Proposition~\ref{thm:tightthm} tells us that the laws of $Y^{(n)}$ are tight on~$\DD([0,T])$. By Proposition~\ref{prop-FDD} we then conclude that $Y^{(n)}$ converges in law to~$\wt B$
while the time-change argument in \eqref{E:3.66q} and \eqref{E:3.69q}
then shows that $t\mapsto \frac1{\sqrt n}Z_{\lfloor tn\rfloor}$, as an element of~$\DD([0,T])$, tends in law to a centered Brownian motion with covariance~$\Sigma$. As the limit process has continuous paths, this implies the convergence of the linear interpolation~$B^n$ of~$Z$-values from \eqref{E:2.1a} in the space $\CC([0,T])$.
\end{proofsect}

\subsection{Assumption \ref{Asmp2} for long-range percolation}
To complete our results concerning quenched invariance principles, it remains to verify the conditions on long-range percolation model that ensure convergence of the random walk to Brownian motion.

\begin{proofsect}{Proof of Corollary~$\ref{cor-2.3}$}
Let~$p>\frac d2$ be as in the statement. Fix any total order~$x\preceq y$ on~$\Z^d$ and let~$\{C_{x,y}\colon x,y\in\Z^d,\, x\preceq y\}$ be independent, zero-one valued random variables with $\BbbP(C_{x,y}=1)=\frakp(x-y)$,
where~$\frakp$ is as in the statement. Identify $C_{x,y}=C_{y,x}$ to get symmetric conductances. Given $n_0\ge1$ to be determined later, let
\begin{equation}
M_n:=\sum_{n_0\le |x|\le n_0+n} |x |^2 (C_{0,x}-\E C_{0,x})=\sum_{n_0\le |x|\le n_0+n} |x|^2 W(x),
\end{equation} where
\begin{equation}
W(x):=C_{0,x}-\E C_{0,x}.
\end{equation}
Then $\{M_n:n\ge1\}$ is a martingale with respect to the filtration~$\FF_n:=\sigma(C_{0,x}\colon|x|\le n_0+n)$ with the variational process
\begin{equation}
\langle M \rangle_n=\sum_{n_0\le |x|\le n_0+n}|x|^4 W(x)^2.
\end{equation}
The Burkholder-Gundy-Davis inequality thus shows, for any~$p\ge1$,
\begin{equation}
\E \bigl(|M_n|^p\bigr)\le c_1\E \bigl(\langle M \rangle_n^{p/2}\bigr)= c_1\E \left[\left(\sum_{n_0\le |x|\le n_0+ n}|x|^4 W(x)^2\right)^{p/2}\right].
\end{equation}
Furthermore, according to \cite[Theorem 1]{Lat}, for every $n\ge 1$,
\begin{multline}
\label{e3}
\quad
\E\left[\left(\sum_{n_0\le |z|\le n_0+ n}|z|^4 W(z)^2\right)^{p/2}\right]\\
\le c_2\inf\left\{t>0\colon\sum_{n_0\le |z|\le n_0+n}\log\left(\E\left[\left(1+t^{-1}|z|^4W(z)^2\right)^{p/2}\right]\right)\le p/2\right\}.
\quad
\end{multline}
We now have to estimate the infimum.

Since $-\frakp(x)\le W(x)\le \IA_{\{C_{0,x}=1\}}$, we have $\E(|W(x)|^r)\le\frakp(x)$ for all~$r\ge1$. By $(1+x)^r\le (1+a x\IA_{\{r\ge1\}}+bx^r)$ valid with $r$-dependent~$a,b>0$ for all~$x>0$, we then get that for $p\ge1$,
\begin{equation}
\begin{aligned}
 \log\left(\E\left[\left(1+t^{-1}|z|^4W(z)^2\right)^{p/2}\right]\right)
\le & \log\left(1+c_3t^{-1} |z|^{4}\frakp(z)\IA_{\{p/2\ge1\}}+c_3t^{-p/2}|z|^{2p}\frakp(z)\right)\\
\le& c_3t^{-1} |z|^{4}\frakp(z)\IA_{\{p/2\ge1\}}+c_3t^{-p/2}|z|^{2p}\frakp(z)\\
\le& c_3(t^{-1}+t^{-p/2})|z|^{2p}\frakp(z),
\end{aligned}\end{equation}
where the constant $c_3$ in the first inequality depends on $p$, in the second inequality we also used the fact that $\ln (1+x)\le x$ for all $x>0$,
and the last inequality is due to $|x|^{4}\IA_{\{p/2\ge1\}}\le |x|^{2p}$ for $|x|\ge 1$.
By our assumption, the sum on the right of \eqref{e3} is bounded by a constant times
\begin{equation}
(t^{-1}+t^{-p/2})\sum_{|x|\ge n_0}|z|^{2p}\frakp(z)
\end{equation}
uniformly in~$n\ge1$.
For a given~$t>0$, say~$t:=1$, this can be made smaller than~$p/2$ by choosing~$n_0$ sufficiently large. The infimum \eqref{e3} is then bounded by one and so $\sup_{n\ge1}\E[|M_n|^p]\le c_1c_2$. With the help of the 
Monotone Convergence Theorem we then get $\hn(0)\in L^p(\BbbP)$. The QIP then follows from Theorem \ref{thm:1}.
\end{proofsect}

As noted earlier, Corollary~\ref{cor-2.3} readily deals with the cases when $\{C_{x,y}=C_{y,x}\}_{x,y\in \Z^d}$ are independent, zero-one valued random variables with (assuming~$x\ne y$)
\begin{equation}
\BbbP(C_{x,y}=1)=\frac{1}{|x-y|^{d+s}}.
\end{equation}
A QIP is then inferred for all~$s>d$. Another example is motivated by long range stable-like random conductance models studied in \cite{CKW}. There one takes (assuming again that~$x\ne y$)
\begin{equation}
C_{x,y}:=\frac{\xi_{x,y}}{|x-y|^{d+s}}
\end{equation}
where $\{\xi_{x,y}=\xi_{y,x}\}_{x,y\in \Z^d}$ are i.i.d.\ Bernoulli random variables except for $|x-y|=1$ where we set $\xi_{x,y}:=1$. In this case the conditions of Corollary~\ref{cor-2.3} are met for all~$s>2$.

\section{Failures of everywhere sublinearity}
\label{sec5}\nopagebreak
\noindent
In this section we provide the promised counterexamples to everywhere sublinearity of the corrector and thus prove Theorems~\ref{thm-2.6} and~\ref{thm-2.8}. We begin with the counterexample arising in the context of long-range percolation.

\subsection{Long-range percolation}
 Consider long-range percolation with the connection probability~$\frakp(x)$ having the asymptotic \eqref{E:2.7} with exponent $s\in(d+2,2d)$, which is non-vacuous only when $d\ge3$. We will assume $\frakp(0)=0$, $\frakp(x)=1$ for~$x$ with $|x|=1$ and~$\frakp(x)<1$ for all~$x$ with~$|x|>1$. The conductances then obey
\begin{enumerate}
\item[(1)] $C_{x,x}=0$ for all~$x$ a.s.,
\item[(2)] $C_{x,y}=1$ whenever $|x-y|=1$ a.s.,
\item[(3)] $\BbbP(C_{x,y}=1)=\frakp(y-x)$ whenever $|x-y|>1$.
\end{enumerate}
As already mentioned, a key point is the proof of the existence of a ``long'' edge of length~$n$ from $o(n)$-neighborhood of the origin. This would itself be easy to guarantee; what makes it harder is that our arguments also need that the ``far away'' endpoint of the ``long'' edge is incident to no other edges than the nearest-neighbor ones. The exact statement is the subject of:

\begin{lemma}\label{e:Lemma5.1}
Noting that $2<s-d<d$ we may pick $\gamma\in(\frac {s-d}{d},\frac{2s-d}{2d}\wedge 1)$ and consider the event
\begin{equation}
\label{E:5.2c}
A(x,y):=\{C_{x,y}=1\}\cap\bigl\{\forall z\in\Z^d\smallsetminus\{x\}\colon |y-z|>1\,\,\Rightarrow\,\, C_{yz}=0\bigr\}.
\end{equation}
Then
\begin{equation}
\label{E:5.2qq}
A_n:=\bigcup_{\begin{subarray}{c}
x\in\Z^d\\|x|\le n^\gamma
\end{subarray}}\,\,
\bigcup_{\begin{subarray}{c}
y\in\Z^d\\n<|y|\le 2n
\end{subarray}}A(x,y)
\end{equation}
occurs for infinitely many~$n$ a.s.
\end{lemma}

\begin{proofsect}{Proof}
Instead of \eqref{E:5.2c} consider the event
\begin{equation}
\wt A(x,y):=\{C_{x,y}=1\}\cap\bigl\{\forall z\in\Z^d\colon |y-z|>1\,\,\&\,\,|z|>n^\gamma\,\,\Rightarrow\,\, C_{yz}=0\bigr\}
\end{equation}
whose advantage over~$A(x,y)$ is that the two events on the right are now independent as soon as~$x$ and~$y$ are as in the union in \eqref{E:5.2qq}. Set
\begin{equation}
\wt A_n:=\bigcup_{\begin{subarray}{c}
x\in\Z^d\\|x|\le n^\gamma
\end{subarray}}\,\,
\bigcup_{\begin{subarray}{c}
y\in\Z^d\\n<|y|\le 2n
\end{subarray}}\wt A(x,y).
\end{equation}
Obviously, $A_n\subset \wt A_n$. Moreover, for~$n$ so large that $n^\gamma<n-1$ (note that~$\gamma<1$), on $A_n^\cc \smallsetminus \wt A_n^\cc$ there is an edge between some $x$ with $|x|\le n^\gamma$ and some~$y$ with $n\le|y|\le 2n$ so that~$y$ has another edge to some $x'$ with $|x'|\le n^\gamma$. Defining, also for later use,
\begin{equation}
B_n:=\left\{\exists x,x',y,y'\in\Z^d\colon\,
\begin{aligned}
&|x|,|x'|\le n^\gamma,\, n\le |y|,|y'|\le 2n
\\
&(x,y)\ne(x',y'),\,C_{x,y}=1=C_{x',y'}
\\
&\,y\ne y'\,\Rightarrow\, C_{y,y'}=1
\end{aligned}
\right\},
\end{equation}
we thus have
\begin{equation}
\label{E:5.6aa}
A_n^\cc\subseteq (\wt A_n^\cc\cap B_n^\cc)\cup B_n.
\end{equation}
We will now proceed to estimate probabilities of two events on the right-hand side.

For the probability of~$B_n$, we invoke a straightforward union bound. Let $\Xi_n$ denote the set of all quadruples $(x,x',y,y')$ that satisfy the geometrical conditions in event~$B_n$. Then, for some constants~$c,c'<\infty$,
\begin{equation}
\label{E:5.7cc}
\begin{aligned}
\BbbP(B_n)
&\le\sum_{x,x',y,y'\in\Xi_n} \frakp(y-x)\frakp(y'-x')\bigl(\delta_{y,y'}+\frakp(y-y')\bigr)
\\
&\le c n^{2d\gamma-2s+o(1)}\!\!\!\!\!\sum_{\begin{subarray}{c}
y,y'\in\Z^d\\n\le|y|,|y'|\le 2n
\end{subarray}}
\bigl(\delta_{y,y'}+\frakp(y-y')\bigr)\le c'n^{2d\gamma-2s+d+o(1)},
\end{aligned}
\end{equation}
where we first used that both $\frakp(y-x)$ and $\frakp(y'-x')$ are at most $n^{-s+o(1)}$, then carried out the sums over~$x$ and~$x'$ to get a constant times~$n^{d\gamma}$ from each and, finally, applied that $z\mapsto \frakp(z)$ is summable because $s>d$. Noting that, in light of our choice of~$\gamma$, the final exponent in \eqref{E:5.7cc} is negative, we get that $B_{2^n}$ occurs only for finitely many~$n$, a.s.

Concerning the first event in \eqref{E:5.6aa}, let~$N$ denote the number of edges between some~$x$ with $|x|\le n^\gamma$ and some~$y$ with $n\le|y|\le 2n$ and let $\{(x_i,y_i)\colon i=1,\dots,N\}$ list the corresponding pairs of vertices connected by these edges. On $\wt A_n^\cc\cap B_n^\cc$ we then know that (once $N>1$) all $y_i$ are distinct and each~$y_i$ must have at least one non-nearest neighbor edge to a vertex~$z$ with $|z|>n^\gamma$ and $z\not\in\{y_1,\dots,y_N\}$. Conditioning on $\FF_n:=\sigma(C_{x,y}\colon |x|\le n^\gamma,\,n\le|y|\le2n)$, we thus have
\begin{equation}
\BbbP\bigl(\wt A_n^\cc\cap B_n^\cc\big|\FF_n\bigr)
\le\IA_{\{N=0\}}+\IA_{\{N>0\}}\,\,\prod_{j=1}^N\biggl(1-\!\!\prod_{\begin{subarray}{c}
y\ne y_1,\dots,y_N\\|y-y_j|>1
\end{subarray}}
\bigl(1-\frakp(y-y_j)\bigr)\,\biggr),
\end{equation}
where~$N$ and $(y_1,\dots,y_n)$ are as specified above. The product is bounded from below by
\begin{equation}
c:=\prod_{|z|>1}(1-\frakp(z))
\end{equation}
which is positive by the summability of~$\frakp$ and our assumption that $\frakp(z)<1$ once~$|z|>1$. Hence,
\begin{equation}
\label{E:5.10}
\BbbP\bigl(\wt A_n^\cc\cap B_n^\cc\bigr)\le \BbbP(N\le n^\delta)+(1-c)^{n^\delta}
\end{equation}
holds true for any $\delta>0$. To estimate $\BbbP(N\le n^\delta)$, let
\begin{equation}
\tilde q_n:=\min_{|x|\le n^\gamma}\min_{n\le|y|\le 2n} \frakp(y-x).
\end{equation}
and let $V_n$ be the number of pairs $(x,y)$ with $|x|\le n^\gamma$ and $n\le|y|\le 2n$. Then~$N$ is stochastically dominated from below by a binomial random variable with parameters $V_n$ and $\tilde q_n$. As $V_n\tilde q_n=n^{d(1+\gamma)-s+o(1)}$ with $d(1+\gamma)-s>0$ by our assumptions about~$\gamma$, the probability $\BbbP(N\le n^\delta)$ decays, for~$\delta$ positive but small, exponentially in a power of~$n$. Using this in \eqref{E:5.10}, the Borel-Cantelli lemma implies that $\wt A_n^\cc\cap B_n^\cc$ occurs only finitely often a.s.
\end{proofsect}

With the existence of the desired ``long'' edge established, we can move to the construction of a counterexample to everywhere sublinearity of the corrector.

\begin{proofsect}{Proof of Theorem~\ref{thm-2.6}}
Consider the long-range percolation setting as specified above.
The asymptotic \eqref{E:2.7} with~$s>d+2$ implies $\E(\sum_{x\in\Z^d}C_{0,x}|x|^2)<\infty$ and so the corrector can be defined by any of the standard methods (see, e.g., Biskup~\cite[Section~3]{Biskup-review} for a discussion of these). In fact, by 
\eqref{E:sub}
above,
the corrector is sublinear on average (cf.~\cite[Proposition~4.15]{Biskup-review}), meaning that $\{x\colon|\chi(x)|>\epsilon|x|\}$ is, for each~$\epsilon>0$, a set of zero density in~$\Z^d$.

To show that~$\chi$ is not sublinear everywhere in the sense of \eqref{E:1.9} we will assume, for the sake of contradiction, that for each~$\epsilon>0$ there is a (random)~$K<\infty$ such that
\begin{equation}
\label{E:chi-bd}
|\chi(x)|\le K+\epsilon|x|,\qquad x\in\Z^d.
\end{equation}
(This is equivalent to \eqref{E:1.9}.) Suppose that~$A_n$ occurs and let~$x$ and~$y$ be the endpoints of an edge that make $A(x,y)$ in the definition of~$A_n$ occur. The harmonicity condition \eqref{E:2.6} for~$\Psi$ from~\eqref{E:2.5} at point~$y$ then reads
\begin{equation}
x+\chi(x)-\bigl(y+\chi(y)\bigr)+\sum_{z\colon|z|=1}\bigl(z+\chi(y+z)-\chi(y)\bigr)=0,
\end{equation}
where we noted that $C_{x'y}=1$ for~$x'=x$ and~$x'$ being a neighbor of~$y$; otherwise~$C_{x'y}=0$. Applying \eqref{E:chi-bd} and the fact that $|x|,|y|,|y+z|\le 2n+1$ for all~$z$ with~$|z|=1$ yield
\begin{equation}
|y-x|\le (2+4d)K+2d+\epsilon(2+4d)(2n+1).
\end{equation}
For~$\epsilon$ small this contradicts $|y-x|>n-n^\gamma$. Hence, by Lemma \ref{e:Lemma5.1}, \eqref{E:chi-bd} cannot occur on~$A_n$ for~$n$ large enough and, since~$A_n$ does occur
 for infinitely many~$n$ a.s., \eqref{E:chi-bd} fails~a.s.
\end{proofsect}

\subsection{Nearest-neighbor conductances}
Next we move to the context underlying Theorem~\ref{thm-2.8}. We start by defining some auxiliary processes that will be  used later to construct the desired environment law~$\BbbP$. As all of these live on the same probability space, we will keep using the same~$\BbbP$ throughout. In the construction we assume that~$d\ge2$ although the ultimate conclusion will be restricted
to $d\ge 3$.

Let $\{\xi_L(x)\colon L\ge1,x\in\Z^d\}$ be independent 0-1-valued random variables with\
\begin{equation}
\BbbP(\xi_L(x)=1)=L^{-d}.
\end{equation}
Note that~$\xi_1(x)=1$ a.s.\ for all~$x$.
Consider
a strictly increasing sequence $\{L_k:k\ge1\}$ of integers with $L_1=1$. A simple use of the Borel-Cantelli lemma shows
\begin{equation}
\sum_{k\ge1}L_k^{-d}<\infty\quad\Rightarrow\quad\sup\bigl\{k\ge1\colon\xi_{L_k}(x)=1\bigr\}<\infty \text{ a.s. }\forall x\in\Z^d.
\end{equation}
(The set on the right is non-empty a.s.\ as $\xi_1(x)=1$ a.s.)
Thus, assuming henceforth $L_k^{-d}$ to be summable, let~$\ell(x)$ denote the maximal~$k$ with $\xi_{L_k}(x)=1$.

Next let $\hate_1,\dots,\hate_d$ be the unit vectors in the coordinate directions and let us regard~$\Z^{d-1}$ as the integer span of $\{\hate_2,\dots,\hate_d\}$. Denote by
\begin{equation}
\Lambda_L:=\bigl\{j\hate_1+z\colon j=-3L,\dots,3L,\,z\in\Z^{d-1},\,|z|\le1\bigr\}\smallsetminus\{0\}.
\end{equation}
the set consisting of~$6L$ vertices in the first coordinate direction and centered at, but not containing, the origin along with all of their nearest neighbors in the other coordinate directions.
Note that
\begin{equation}
\BbbP(\ell(x)\ge j)=1-\prod_{k\ge j}(1-L_k^{-d})\le
\sum_{k\ge j}
L_k^{-d}.
\end{equation}
By the monotonicity of~$k\mapsto L_k$, we have
$\sum_{j\ge 1}L_j\sum_{k\ge j}L_k^{-d}
=\sum_{k\ge 1}\sum_{j=1}^kL_jL_k^{-d}\le \sum_{k\ge1}kL_k^{1-d}$,
so another use of the Borel-Cantelli lemma
gives
\begin{equation}
\sum_{k\ge1}kL_k^{1-d}<\infty
\quad\Rightarrow\quad\sup\bigl\{k\ge1\colon\max_{z\in\Lambda_{L_k}}\ell(x+z)\ge k\bigr\}<\infty \text{ a.s. }\forall x\in\Z^d.
\end{equation}
Assuming henceforth $kL_k^{1-d}$ to be summable, let~$m(x)$ denote the maximal~$k$ in this set for the given~$x$. As~$\{L_k:k\ge1\}$ is increasing, we get
\begin{equation}
\label{E:4.10}
m(x)<\ell(x)\quad\Rightarrow\quad m(x+z)\ge\ell(x)>\ell(x+z),\qquad z\in\Lambda_{L_{\ell(x)}}.
\end{equation}
Obviously, the collection $\{(\ell(x),m(x))\colon x\in\Z^d\}$ is stationary. Moreover, as~$\Lambda_L$ does not contain the origin, $m(x)$ is independent of~$\ell(x)$ for each~$x$. We now observe:

\begin{lemma}
\label{lemma-4.1}
Suppose that $\sum_{k\ge1}kL_k^{1-d}<\infty$. Then there is a constant~$c>0$ such that
\begin{equation}
\label{E:4.11}
cL_k^{-d}\le \BbbP\bigl(m(x)<\ell(x)=k\bigr)\le L_k^{-d}
\end{equation}
holds true for all
$k\ge2$ and all~$x\in\Z^d$.
Moreover, if also $L_{k+1}>2L_k$ for all $k\ge1$, then
\begin{equation}
\label{E:5.12a}
\bigl\{\exists x\in\Z^d\colon  L_k\le|x|_\infty\le 2L_k,\,m(x)<\ell(x)=k\bigr\}
\end{equation}
occurs for infinitely many~$k$, a.s.
\end{lemma}

\begin{proofsect}{Proof}
The definition of~$\ell(x)$ gives
\begin{equation}
\BbbP\bigl(\ell(x)=k\bigr)=L_k^{-d}\prod_{j>k}(1-L_j^{-d}).
\end{equation}
This yields immediately the upper bound in \eqref{E:4.11}.
On  the  other hand, the fact that $\{L_k\}$ is non-decreasing shows
\begin{equation}
\BbbP\bigl(m(x)<k\bigr)=\biggl(\,\prod_{j\ge k}(1-L_j^{-d})\biggr)^{|\Lambda_{L_k}|}\,\prod_{j>k}\Biggl(\,\biggl(\,\prod_{r\ge j}(1-L_r^{-d})\biggr)^{|\Lambda_{L_j}\smallsetminus\Lambda_{L_{j-1}}|}\Biggr).
\end{equation}
By $|\Lambda_L|=O(L)$, the fact that $L_2>1$ and the summability of $kL_k^{1-d}$, both terms in the parentheses are bounded from below by a positive constant uniformly in $k\ge2$. Since $m(x)$ and~$\ell(x)$ are independent we get the lower bound in \eqref{E:4.11} as well.

For the second part, recall that we regard~$\Z^{d-1}$ as the linear span of~$\{\hate_2,\dots,\hate_d\}$ over the ring of integers. Given $y\in\Z^{d-1}$ and $j\in\Z$, define
\begin{equation}
G_k(y,j):=\bigl\{m(y+j\hate_1)<\ell(y+j\hate_1)=k\bigr\}.
\end{equation}
We observe that, by \eqref{E:4.10}, we have $G_k(y,j)\cap G_k(y,j')=\emptyset$ as long as $0<|j-j'|\le 3L_k$. Hence, invoking also the lower bound in \eqref{E:4.11}, we get
\begin{equation}
\BbbP\Bigl(\bigcup_{j={L_k}}^{2L_k}G_k(y,j)\Bigr)
=\sum_{j=L_k}^{2L_k}\BbbP\bigl(m(y+j\hate_1)<\ell(y+j\hate_1)=k\bigr)
\ge c L_k^{1-d}.
\end{equation}
Moreover, the giant unions are for distinct $y\in(3\Z)^{d-1}$ independent. Hence, we get
\begin{equation}
\label{E:5.16a}
\BbbP\biggl(\,\bigcup_{\begin{subarray}{c}
y\in(3\Z)^{d-1}\\L_k\le|y|_\infty\le 2L_k
\end{subarray}}
\bigcup_{j={L_k}}^{2L_k}G_k(y,j)\biggr)\ge c'>0
\end{equation}
for some $c'$ independent of~$k$.

Now observe that the union in \eqref{E:5.16a} is a subset of the event \eqref{E:5.12a}. Also note that, as soon as we have $L_{k+1}>2L_k$, the unions in \eqref{E:5.16a} use, for distinct $k$'s, disjoint sets of underlying coordinates~$\{\xi_L(x)\colon L\ge1,\,x\in\Z^d\}$ and are thus  independent of one another. By the second Borel-Cantelli lemma, the event in \eqref{E:5.12a} occurs for infinitely many~$k$ a.s.
\end{proofsect}

Let us introduce the shorthand
\begin{equation}
\label{E:eta-def}
\kappa(x):=\IA_{\{m(x)<\ell(x)\}}
\end{equation}
and note that $\{\kappa(x)\}$ is a stationary, ergodic process with a positive density of 1's.
The following observation will turn out to be quite useful:

\begin{lemma}
\label{lemma-4.2}
Given $x\in\Z^d$, let $E_L(x)$ denote the set of (nearest-neighbor) edges incident with vertices in $\{x+j\hate_1\colon j=0,\dots,L\}$. Then
\begin{equation}
x\ne\tilde x\quad\&\quad\kappa(x)=1=\kappa(\tilde x)
\quad\Rightarrow\quad E_{L_{\ell(x)}}(x)\cap E_{L_{\ell(\tilde x)}}(\tilde x)=\emptyset.
\end{equation}
\end{lemma}

\begin{proofsect}{Proof}
If $E_{L_{\ell(x)}}(x)\cap E_{L_{\ell(\tilde x)}}(\tilde x)\ne\emptyset$ and $x\ne\tilde x$, then $x\in \tilde x+\Lambda_{L_{\ell(\tilde x)}}$ and $\tilde x\in x+\Lambda_{L_{\ell(x)}}$. But then \eqref{E:4.10} and \eqref{E:eta-def} yield
$m(x)\ge\ell(\tilde x)>\ell(x)$ and, similarly, $m(\tilde x)\ge\ell(x)>\ell(\tilde x)$, a contradiction.
\end{proofsect}

\begin{proofsect}{Proof of Theorem~\ref{thm-2.8}}
Let $p,q\ge1$ be numbers such that \eqref{E:2.8} holds and let $p'>p$ and $q'>q$ be such that we still have
\begin{equation}
\label{E:5.19}
\frac1{p'}+\frac1{q'}>\frac2{d-1}.
\end{equation}
(This is where we need to require $d\ge 3$.) Define sequences
\begin{equation}
\label{E:5.20}
a_L:=L^{-(d-1)/q'}\quad\text{and}\quad b_L:=L^{(d-1)/p'}.
\end{equation}
Consider the construction given above with~$\{L_k:k\ge1\}$ such that $L_{k+1}>2L_k$ and $L_1:=1$ so that all objects $\ell(x)$, $m(x)$ and $\kappa(x)$ are well defined. Given an~$x$ with $\kappa(x)=1$, denote $k:=\ell(x)$ and consider the set of edges $E_{L_k}(x)$ incident with at least one vertex in~$\{x+j\hate_1\colon j=0,\dots, L_k\}$. Set the conductance to $b_{L_k}$ on edges with both endpoints in this set and to $a_{L_k}$ to those with only one endpoint in this set. Thanks to Lemma~\ref{lemma-4.2}, the conductance of each edge is set at most once so no conflict can arise. We set the conductance on edges not in $\bigcup \{E_{L_{\ell(x)}}\colon \kappa(x)=1\}$ to one.

The resulting configuration of conductances is a measurable function of $\{\xi_L(x)\colon L\ge1,\,x\in\Z^d\}$ and, since this family is stationary and ergodic with respect to shifts, so is the induced conductance law. Let us check that the integrability conditions \eqref{E:2.9a} hold. Fix any~$x$ with~$|x|=1$.
Noting that $E_{L_k}(z)$ contains $L_k$ edges of conductance $b_{L_k}$ and $R_k:=2+(2d-2)(L_k+1)$ edges of conductance~$a_{L_k}$, we have
\begin{equation}
\E (C_{0,x}^p)\le 1+\sum_{k\ge1}L_k^{-d}\bigl(\,L_k(b_{L_k})^p+R_k\bigl(a_{L_k})^p\bigr).
\end{equation}
Plugging in \eqref{E:5.20}, invoking that $p'>p$ and~$q'>q$ and using that $\{L_k\}$ grows exponentially, we get $C_{0,x}\in L^p(\BbbP)$ as desired. Similarly,
\begin{equation}
\E (C_{0,x}^{-q})\le 1+\sum_{k\ge1}L_k^{-d}\bigl(\,L_k(b_{L_k})^{-q}+R_k\bigl(a_{L_k})^{-q}\bigr),
\end{equation}
which is again finite by \eqref{E:5.20}, our choices of~$p'$ and $q'$ and the exponential growth of the sequence~$\{L_k\}$.

Now let us move to the violation of sublinearity of the corrector. Suppose the event \eqref{E:5.12a} occurs at some $x$ with $L_k\le|x|_\infty\le 2L_k$. The conductances $C_{yz}$ on edges $\langle y,z\rangle\in E_{L_k}(x)$ then take values $a_{L_k}$ and $b_{L_k}$ as specified above.
Denote by  $D:=\{x+j\hate_1\colon j=0,\dots,L_k\}$
the corresponding set of vertices (which depends on~$x$) and let
\begin{equation}
\EE_D(f):=\sum_{\langle y,z\rangle\in E_{L_k}(x)}\!C_{yz}\,\bigl|f(y)-f(z)\bigr|^2
\end{equation}
be the Dirichlet energy for a ($\R^d$-valued) function~$f$ on~$D$. The ``harmonic coordinate''~$\Psi$ from \eqref{E:2.5} solves the Dirichlet problem on~$D$ and so $f:=\Psi$ has minimal~$\EE_D(f)$ among all functions that agree with $\Psi$ on the external boundary~$\partial D$ of~$D$.

We now derive bounds on $\EE_D(\Psi)$. To get a lower bound, we fix the values at~$x$ and~$x+L_k\hate_1$ and set all conductances on edges with only one endpoint in~$D$ to zero. Optimizing the remaining values is now a one-dimensional problem whose simple solution yields
\begin{equation}
\EE_D(\Psi)\ge b_{L_k}L_k^{-1}\bigl|\Psi(x+L_k\hate_1)-\Psi(x)\bigr|^2.
\end{equation}
For the upper bound, we take the test function $f$ that equals $\Psi(x)$ everywhere on~$D$. This gives
\begin{equation}
\EE_D(\Psi)\le a_{L_k}\,\sum_{y\in\partial D}\bigl|\Psi(y)-\Psi(x)\bigr|^2.
\end{equation}
Let us now see that this is not compatible with sublinearity of the corrector. Indeed, if \eqref{E:chi-bd} were true, then the fact that $D\cup\partial D\subset[-3L_k,3L_k]^d$ yields
\begin{equation}
\bigl|\Psi(y)-\Psi(x)\bigr|\le (1+6\epsilon)L_k+2K,\qquad y\in\partial D,
\end{equation}
while
\begin{equation}
\bigl|\Psi(x+L_k\hate_1)-\Psi(x)\bigr|\ge (1-6\epsilon)L_k-2K.
\end{equation}
But that contradicts the fact, implied by \eqref{E:5.19}, that $a_{L_k}L_k^2|\partial D|\ll b_{L_k} L_k$ once $k$ is sufficiently large. Hence we cannot have \eqref{E:chi-bd} and, at the same time, the event in \eqref{E:5.12a} to occur for~$k$ large. Lemma~\ref{lemma-4.1} implies that \eqref{E:chi-bd} fails for all~$\epsilon>0$ and all~$K<\infty$ a.s.
\end{proofsect}

\section*{Acknowledgments}
\noindent
 We thank
anonymous referees for their helpful comments and corrections.
This research has been supported by National Science Foundation (US) awards DMS-1712632 and DMS-1954343, the
National Natural Science Foundation of China (No.\ 11871338), JSPS KAKENHI Grant Number JP17H01093, the Alexander von Humboldt Foundation, the National
Natural Science Foundation of China (Nos.\ 11831014 and 12071076),
the Program for Probability and Statistics: Theory and Application (No.\ IRTL1704) and the Program for Innovative Research Team in Science and Technology in Fujian Province University (IRTSTFJ). The non-Kyoto based authors would like to thank RIMS at Kyoto University for hospitality that made this project possible.

A version of this paper by two of the present authors was previously posted on arXiv~\cite{BK-failed} and was subsequently withdrawn due to an error in one of the key arguments. The present paper subsumes the parts of~\cite{BK-failed} that are worth saving.

%
%\vskip 0.3truein
%{\small
%{\bf Marek Biskup:}
%   Department of Mathematics, UCLA, Los Angeles, California, USA.
%   \newline \texttt{biskup@math.ucla.edu}
%
%\bigskip
%
%{\bf Xin Chen:}
%   Department of Mathematics, Shanghai Jiao Tong University, 200240 Shanghai, P.R. China. \texttt{chenxin217@sjtu.edu.cn}
%	
%\bigskip
%	
%
%{\bf Takashi Kumagai:}
% Research Institute for Mathematical Sciences,
%Kyoto University, Kyoto 606-8502, Japan.
%\texttt{kumagai@kurims.kyoto-u.ac.jp}
%
%\bigskip
%
%{\bf Jian Wang:}
%    College of Mathematics and Informatics \& Fujian Key Laboratory of Mathematical Analysis and Applications, Fujian Normal University, 350007 Fuzhou, P.R. China.
%     \texttt{jianwang@fjnu.edu.cn}}
		
\end{document}